\documentclass[a4paper,11pt]{amsart}
\usepackage{tikz,tikz-cd,amscd,amsmath,amsthm,amssymb,braket,mathrsfs,graphicx,verbatim,mathtools}
\usepackage{hyperref}
\hypersetup{colorlinks=true, linkcolor=black,citecolor=black}
\usepackage{xcolor}
\definecolor{revised}{RGB}{0,60,120}

\usepackage{setspace}

\usepackage[margin=3cm]{geometry}

\makeatletter
\newtheorem*{rep@theorem}{\rep@title}
\newcommand{\newreptheorem}[2]{%
\newenvironment{rep#1}[1]{%
 \def\rep@title{#2 \ref{##1}}%
 \begin{rep@theorem}}%
 {\end{rep@theorem}}}
\makeatother

\let\oldtocsection=\tocsection
 
\let\oldtocsubsection=\tocsubsection
 
\let\oldtocsubsubsection=\tocsubsubsection
 
\renewcommand{\tocsection}[2]{\hspace{0em}\oldtocsection{#1}{#2}\vspace*{-1.5mm}}
\renewcommand{\tocsubsection}[2]{\hspace{1em}\oldtocsubsection{#1}{#2}\vspace*{-1.5mm}}
\renewcommand{\tocsubsubsection}[2]{\hspace{2em}\oldtocsubsubsection{#1}{#2}\vspace*{-1.5mm}}

\newreptheorem{theorem}{Theorem}
\newtheorem{theorem}{Theorem}[section]
\theoremstyle{definition}\newtheorem*{definition}{Definition}
\theoremstyle{definition}\newtheorem{proposition}[theorem]{Proposition}
\theoremstyle{definition}\newtheorem{lemma}[theorem]{Lemma}
\theoremstyle{definition}\newtheorem{corollary}[theorem]{Corollary}
\theoremstyle{definition}
\theoremstyle{remark}\newtheorem{remark}[theorem]{Remark}

\newcommand{\length}{\mathcal{L}}
\newcommand{\s}{{\tilde s}}
\newcommand{\R}{\mathbb{R}}
\newcommand{\I}{\mathcal{I}}

\DeclareMathOperator{\dist}{dist}
\DeclareMathOperator{\sgn}{sgn}

\DeclareMathOperator{\grad}{grad}
\newcommand{\htwospace}{H^2(S^1,\R^n)}
\newcommand{\norm}[1]{{\left \Vert {#1}\right \Vert}}

\newcommand{\ltwo}[1]{{\left \Vert {#1}\right\Vert}_{L^2}}

\newcommand{\htwo}[1]{{\left \Vert {#1}\right\Vert}_{H^2}}

\newcommand{\abs}[1]{{\left \vert {#1}\right \vert}}
\newcommand{\ip}[1]{\left\langle {#1} \right\rangle }
\DeclareMathOperator{\E}{\mathcal{E}}

\DeclareMathOperator{\op}{\rm{op}}
\DeclareMathOperator{\rec}{\rm{rec}}

\DeclareMathOperator{\pr}{\rm{pr}}

\makeatletter
\@namedef{subjclassname@2020}{%
  \textup{2020} Mathematics Subject Classification}
\makeatother

\makeatletter
\@namedef{subjclassname@2020}{%
  \textup{2020} Mathematics Subject Classification}
\makeatother

\allowdisplaybreaks

\begin{document}

\title{Convergence of Sobolev gradient trajectories to elastica}

\author[S. Okabe]{Shinya Okabe }
\address[S. Okabe]{Mathematical Institute, Tohoku University,
Aoba, Sendai 980-8578, Japan.}
\email{\tt shinya.okabe@tohoku.ac.jp}
\thanks{The first author is supported in part by JSPS KAKENHI Grant Numbers JP19H05599, JP20KK0057 and JP21H00990.}

\author[P. Schrader]{Philip Schrader}
\address[P. Schrader]{Mathematical Institute, Tohoku University,
Aoba, Sendai 980-8578, Japan.}
\email{\tt philschrad@gmail.com}

\thanks{The second author is supported by an International Postdoctoral Research Fellowship of the Japan Society for the Promotion of Science and JSPS KAKENHI Grant Number JP19F19710.}



\date{}
\maketitle

\begin{abstract}
In this paper we study the $H^2(ds)$-gradient flow for the modified elastic energy defined on closed immersed curves in $\mathbb{R}^n$. 
We prove the existence of a unique global-in-time solution to the flow and establish full convergence to elastica by way of a \L ojasiewicz--Simon gradient inequality. \\

\end{abstract}

\maketitle

\tableofcontents

\section{Introduction}

This paper is concerned with a Sobolev gradient flow for the modified elastic energy defined on closed immersed curves $\gamma : S^1 \to \R^n$: 
\begin{equation} \label{energy}
	\E(\gamma):= \int_\gamma k^2 \, ds + \lambda^2 \length(\gamma), 
\end{equation}
where $n \in \mathbb{N}_{\ge 2}$; $\lambda$ is a nonzero constant; $\length(\gamma)$, $k$ and $s$ denote the length, the curvature and the arc length parameter of $\gamma$, respectively; and throughout the paper $S^1$ is identified with $[0,1]$ modulo endpoints.
In the case $n=2$ the critical points of $\E$ are the classical Euler--Bernoulli elastica.

 A variety of gradient flows toward elastica have been studied by other authors,
beginning with the curve-straightening flow studied by Langer and Singer for closed curves in $\R^3$ \cite{Langer-Singer_1985} and Riemannian manifolds \cite{Langer-Singer_1987}. Langer and Singer consider the restriction of the total squared curvature to constant speed curves. Under this restriction the total squared curvature of a curve is equivalent to the Dirichlet energy of its tangent indicatrix. Langer and Singer show that the Dirichlet energy satisfies the Palais--Smale condition on a Hilbert manifold of tangent indicatrices of Sobolev class $H^1$. Consequences of this are existence for all positive time as well as sub-convergence (convergence of a subsequence) of the associated gradient flow, where the gradient is defined by the $H^1$-metric on indicatrices.

The work of Langer and Singer was extended in several directions by Linn\'er
 \cite{Linner_1989,Linner:1991aa} and also Wen who studied the $L^2$-gradient flow of Dirichlet energy on indicatrices \cite{Wen_1993} and then the $L^2(ds)$-gradient flow of the total squared curvature on constant speed planar curves \cite{Wen_1995}. In both cases Wen proved that in the case of smooth initial data with non-zero winding number, solutions exist for all positive time and converge to circles. Around the same time Koiso \cite{Koiso_1996} obtained similar results for unit speed space curves, and Polden \cite{P_1996} considered the $L^2(ds)$-flow of the \emph{modified} elastic energy $\E$ for planar curves without constraint. Polden found that  given smooth initial data the flow exists globally and subconverges modulo translations to critical points of the energy. Dziuk et. al. \cite{DKS_2000}  extended the results of Wen and Polden to closed curves in $\R^n$.

More recently the $L^2(ds)$-gradient flow for $\E$ has become known as the elastic flow and has been studied on spaces of closed curves with area constraint \cite{O_2007,O_2008}, open curves with boundary conditions  \cite{Lin_2012,DLP_2017,Spener_2017}, non-compact curves \cite{NO_2014} and planar networks \cite{dall2020elastic}. For a more complete list of references we refer to the recent survey \cite{MPP_2021}.
Standard results are solvability of the elastic flow for smooth initial curve and subconvergence of solutions to elastica. There is room for improvement on both fronts. With initial data in the energy class $H^2$, well-posedness of the elastic flow with boundary conditions was proved in \cite{RS_2020} by way of analytic semigroup theory; and for the more general $p$-elastic flow (which is the elastic flow when $p=2$) \cite{blatt2022minimising,NP_2020,OPW,OW_2021} show existence of weak solutions by the minimizing movements method.  
Stronger convergence results are obtained in \cite{DallAcqua:2016aa,MP_2021} using \L ojasiewicz--Simon gradient inequalities. We note however that except for some special cases of weak solutions to the $p$-elastic flow, the existence results do not include uniqueness. Moreover, the convergence results are always up to reparametrisation and sometimes translation\footnote{\cite{Koiso_1996,O_2007,O_2008} proved full convergence without any translation corrections 
due to the invariance of center of gravity of curves under the elastic flow with the inextensibility condition, but the inextensibility condition fixes parametrisation. In \cite{MP_2021} translation assumptions are neatly avoided without constraints, but reparametrisation is still required.} . 

In this paper we introduce a new gradient flow for $\E$. We prove the existence of unique global-in-time strong solutions starting from initial curves in the energy space $H^2$ and full convergence of gradient trajectories to elastica without any corrections to translation or parametrisation. 
Namely, we consider the Cauchy problem for the $H^2(ds)$-gradient flow for $\E$ defined on closed curves in $\R^n$: 
\begin{align}  \label{eq:GF} \tag{GF} 
\begin{cases}
& \partial_t \gamma = - \grad \E_\gamma, \\
& \gamma(\cdot, 0)= \gamma_0(\cdot). 
\end{cases}
\end{align}
Here, $\grad \E_\gamma$ denotes the $H^2(ds)$-gradient for $\E$ at $\gamma$ (for the precise definition and its expression, see Section \ref{section:formulation}), i.e. the gradient of $\E$ with respect to the inner product  
\begin{equation}\label{h2ds'}
 \ip{v,w}_{H^2(ds),\gamma}:=\int_\gamma \ip{v,w}\, ds+\int_\gamma \ip{v_s,w_s}\, ds+\int_\gamma\ip{v_{ss},w_{ss}} ds
\end{equation}
for variations $v,w\in H^2(S^1,\R^n)$ along $\gamma$.
We consider initial data $\gamma_0$ in the space of $H^2$ immersions
\[ \I^2(S^1,\R^n):= \{ \gamma\in H^2(S^1,\R^n): |\gamma'(u)|> 0 \} \]
which is an open subset of $H^2(S^1,\R^n)$ (see Lemma \ref{superlemma} (i)). 
The main result of this paper is as follows:

\begin{theorem} \label{maintheorem}
Let $\gamma_0 \in \I^2(S^1,\R^n)$. 
Then problem \eqref{eq:GF} possesses a unique global-in-time solution $\gamma$ in the class $C^1([0,\infty),\I^2(S^1,\R^n))$. 
Moreover, the solution $\gamma$ converges to an elastica as $t \to \infty$ in the $H^2$-topology. 
\end{theorem}

Since it is defined by the $H^2(ds)$-metric, $\grad\E_\gamma$ is also of class $H^2$ and so \eqref{eq:GF} is an ODE in the Sobolev space $H^2(S^1,\R^n)$. 
Thus we prove the existence of local-in-time solutions to problem \eqref{eq:GF} by the use of the generalized (to Banach space) Picard--Lindel\"of theorem (Proposition~\ref{shortexistence}). 
Moreover, thanks to the metric completeness of the space $\I^2(S^1,\R^n)$ with respect to the $H^2(ds)$-Riemannian metric  (see \cite[Theorem 4.3]{Bruveris:2015aa}), 
the proof of the existence of global-in-time solutions follows from a simple adaptation of the method  used in \cite[Theorem 9.1.6]{Palais:1988fv}.

The metric completeness is also a key ingredient in the proof of convergence: we use a \L ojasiewicz--Simon gradient inequality to show that the $H^2(ds)$-length of a gradient trajectory is finite, and therefore converges by completeness. Proving that such an inequality holds for $\E$ is complicated by the fact that, due to reparametrisation invariance, the second derivative of $\E$ has an infinite dimensional nullspace and therefore cannot be Fredholm. This can be overcome by restricting the energy to a submanifold consisting of curves which form a cross-section of the reparametrisation symmetry, proving that a \L ojasiewicz--Simon inequality holds for this restriction, and then extending the inequality by symmetry. This is the method used in e.g. \cite{chill2009willmore},\cite{DallAcqua:2016aa} and \cite{MP_2021} where the chosen submanifold consists of normal graphs over the critical point. Here we instead restrict to the submanifold of arc length proportionally parametrised curves. Even in finite codimension, verifying a \L ojasiewicz-Simon gradient inequality on a submanifold is far from trivial (see \cite{rupp2020lojasiewicz}), but here we are helped by the fact that the restriction of $\E$ to arc length proportionally parametrised curves takes a simple form.

We remark that the $H^2(ds)$-gradient flow for $\E$ is different from the flow considered by Langer and Singer \cite{Langer-Singer_1985,Langer-Singer_1987} and Linn\'er \cite{Linner_1989,Linner_2003}. These authors carry out their analysis using an $H^1$-metric on a space of indicatrices, which is effectively $H^2$ on the space of curves but without measuring the zeroth order product of variations -- the first summand in \eqref{h2ds'}. As observed by Wen \cite{Wen_1995} and then clearly demonstrated by  Linn\'er \cite{Linner_2003}, different representations of the space of curves and different choices for the metric on these spaces result in geometrically distinct flows. 

Finally, we mention that Sobolev gradient flows for other geometric functionals have been studied, e.g. the $H^1(ds)$-curve shortening flow \cite{Schrader:2021aa}, 
the $H^2$-elastic flow for graphs with an obstacle \cite{Muller_2020}, and fractional Sobolev gradient flows for knot energies \cite{Reiter-Schumacher_2020, Knappmann:2021aa}.
We also mention the comprehensive book by Neuberger \cite{Neuberger:2010aa} which discusses many applications as well as the numerical advantages of Sobolev gradient flows. 

The rest of this paper is organized as follows:  
In Section~\ref{section:formulation} we give the precise formulation of the $H^2(ds)$-gradient flow for $\E$ and introduce some notation.  
In Section~\ref{section:existence} we prove the existence of a unique global-in-time solution of \eqref{eq:GF}. 
We prove the full convergence of solutions to \eqref{eq:GF} as follows: 
in Section~\ref{subsection:analyticity} we verify the analyticity of the functional $\E$; 
in Section~\ref{subsection:al-para-curves} we introduce the submanifold of arc length proportionally parametrised curves, and in Section~\ref{subsection:LSineq} we prove a \L ojasiewicz--Simon gradient inequality on this submanifold which we then extend to a \L ojasiewicz--Simon gradient inequality on all of $\I^2(S^1,\R^n)$; 
finally we prove full convergence of solutions to an elastica by way of the \L ojasiewicz--Simon gradient inequality in Section~\ref{subsection:full-limit-convergence}.\\

\noindent\textbf{Acknowledgements.}
The authors are very grateful to the anonymous referees for their careful reading and for many valuable suggestions, and the second author would like to thank Glen Wheeler for helpful discussions. Most of the work in this paper was completed while the second author was a JSPS postdoctoral fellow at Tohoku University, and he wishes to express his gratitude for the kind hospitality of the staff at Tohoku University and in the overseas fellowship division of the JSPS.

\section{Formulation} \label{section:formulation}
\subsection{The $H^2(ds)$-gradient}
We derive the $H^2(ds)$-gradient for the modified elastic energy $\E$ \eqref{energy} defined on closed immersed curves $\gamma \in \I^2(S^1,\R^n)$. 
For $v,w\in \htwospace$ define the $H^2(ds)$-inner product by 
\begin{equation}\label{h2ds}
 \ip{v,w}_{H^2(ds),\gamma}:=\ip{v,w}_{L^2(ds),\gamma}+\ip{v_s,w_s}_{L^2(ds),\gamma}+ \ip{v_{ss},w_{ss}}_{L^2(ds),\gamma} 
\end{equation}
with 
\begin{equation}\label{l2ds}
\ip{v,w}_{L^2(ds),\gamma}:= \int_\gamma \ip{v,w}\, ds,  
\end{equation}
where $\ip{\cdot,\cdot}$ denotes the Euclidean product.
From now on we will omit the subscript $\gamma$ from the $H^2(ds),L^2(ds)$ products unless it is needed. 
Because they depend on the base curve  $\gamma$ the $L^2(ds)$ and $H^2(ds)$ products \eqref{h2ds}, \eqref{l2ds} are \emph{Riemannian metrics} on $\I^2(S^1,\R^n)$.
The $H^2(ds)$-norm is equivalent to the usual $H^2$-norm but the constants $c_1, c_2$ in
\[ c_1\norm{v}_{H^2}\leq \norm{v}_{H^2(ds)}\leq c_2\norm{v}_{H^2}\]
will depend on $\gamma$. Indeed from $v_s=\frac{v'}{\abs{\gamma'}}$ and $v_{ss}=\frac{1}{\abs{\gamma'}^2}v''-\frac{\ip{\gamma'',\gamma'}}{\abs{\gamma'}^4}v'$, setting $c_0=(\min \abs{\gamma'})^{-1}$
\begin{align*}
\norm{v}_{H^2(ds)}^2 &\leq  \norm{\gamma'}_{L^\infty}\norm{v}_{L^2}^2+c_0\norm{v'}^2_{L^2}+c_0^3\norm{v''}^2_{L^2}+c_0^7\norm{\gamma'}_{L^\infty}^2\norm{\gamma''}_{L^2}^2\norm{v'}_{L^\infty}^2\\
& \leq c_2^2\norm{v}_{H^2}^2. 
\end{align*}
The left inequality in the norm equivalence is similar. 

It then follows from the Lax-Milgram theorem that $H^2(ds)$ is a \emph{strong} Riemannian metric, meaning that at each $\gamma$ it gives an continuous linear isomorphism between $H^2(S^1,\R^n)$ and its dual. 
For contrast, the $L^2(ds)$-product does not - it is a \emph{weak} Riemannian metric. 

For the Gateaux derivative of $\E$ at $\gamma\in \I^2(S^1,\R^n)$ we find
\begin{equation}\label{dE}
	d\E_\gamma v=\dfrac{d}{d \varepsilon}\mathcal{E}(\gamma + \varepsilon v) \Bigl|_{\varepsilon=0}
 = \int^{\mathcal{L}(\gamma)}_0 [ 2 \ip{\gamma_{ss} , v_{ss}} - 3 k^2 \ip{\gamma_s, v_s} - \lambda^2 \ip{\gamma_{ss}, v} ] \, ds
\end{equation}
for all $v \in H^2(S^1; \mathbb{R}^n)$. 
Using Lemma \ref{superlemma} it is not too difficult to show that $d\E_\gamma$ is bounded in $(H^2)^*$. Moreover, it depends continuously on $\gamma$ and is therefore a Frech\'et derivative. Then since $H^2(ds)$ is a strong metric there exists an $H^2(ds)$-gradient of $\mathcal{E}$ at $\gamma$, meaning
\begin{equation}
\label{def-H2-EL}
\ip{\grad \E_\gamma, v}_{H^2(ds)} 
 = \int^{\mathcal{L}(\gamma)}_0 [ 2 \ip{\gamma_{ss} , v_{ss}} - 3 k^2 \ip{\gamma_s , v_s} - \lambda^2 \ip{\gamma_{ss} , v} ] \, ds
\end{equation}
for all $v \in H^2(S^1; \mathbb{R}^n)$. 
To derive the explicit form of $\grad \E_\gamma$, observe that \eqref{def-H2-EL}
is the weak formulation of
\begin{equation}\label{gradstrong}
	(\grad\E_\gamma)_{ssss}-(\grad\E_\gamma)_{ss}+\grad\E_\gamma = \nabla \E_\gamma 
\end{equation}
where $\nabla \E_\gamma=2\gamma_{ssss}+3(k^2\gamma_s)_s - \lambda^2 \gamma_{ss}$ is the classical $L^2(ds)$-gradient. 

To solve \eqref{gradstrong} we derive the Green's function, i.e. the solution to 
\begin{equation}\label{eq:Gde}
  G_{xxxx}(x,y)-G_{xx}(x,y)+G(x,y)=\delta(x-y)    
\end{equation}
which is $C^2$ periodic. Using the general solution to the homogeneous equation we set 
\begin{align*}
	&G(x,y)=\\
&\begin{cases}
b_1e^{\frac{\sqrt{3}}{2}x}\cos\tfrac{x}{2}
+b_2e^{\frac{-\sqrt{3}}{2}x}\cos\tfrac{x}{2}
+b_3e^{\frac{\sqrt{3}}{2}x}\sin\tfrac{x}{2}
+b_4e^{\frac{-\sqrt{3}}{2}x}\sin\tfrac{x}{2}
& \text{if} \quad x<y,\\
c_1e^{\frac{\sqrt{3}}{2}x}\cos\frac{x}{2}
+c_2e^{\frac{-\sqrt{3}}{2}x}\cos\frac{x}{2}
+c_3e^{\frac{\sqrt{3}}{2}x}\sin\frac{x}{2}
+c_4e^{\frac{-\sqrt{3}}{2}x}\sin\frac{x}{2} & \text{if} \quad x>y.
	\end{cases}
\end{align*}
In order to solve for the unknown parameters $b_i$ and $c_i$ we enforce periodic boundary conditions $\partial_1^iG(0,y)=\partial_1^iG(\length(\gamma),y)$ for $ i=0,1,2,3$, continuity at $x=y$ of $\partial_1^iG(x,y) $ for $i=0,1,2$, and the third order discontinuity 
\[ \lim_{x\to y^+}\partial_1^3 G(x,y)-\lim_{x\to y^-}\partial_1^3 G(x,y)=1.\] 
We find: 
\begin{equation}\label{eq:defG}
G( x,y;\gamma )=\frac{A (\length(\gamma)-\abs{x-y},|x-y| )}{\beta(\mathcal{L}(\gamma))}, \qquad 0\leq x,y\leq \length(\gamma),
\end{equation}
where 
\begin{align}\nonumber 
A(x_1,x_2)&= \sinh \frac{\sqrt{3} x_1}{2} \cos\frac{x_2}{2}
+\sinh\frac{\sqrt{3}x_2}{2} \cos\frac{x_1}{2} \\  \label{eq:A}
& \qquad +\sqrt{3}\cosh\frac{\sqrt{3} x_1}{2} \sin\frac{x_2}{2}
+\sqrt{3}\cosh\frac{\sqrt{3}x_2}{2} \sin\frac{x_1}{2}, \\ \nonumber 
\beta(\ell) &=2\sqrt{3}\Bigl(\cosh\frac{\sqrt{3}\ell}{2} - \cos\frac{\ell}{2} \Bigr). 
\end{align}
Now, setting
\begin{align*}
s := s_\gamma(u) = \int^u_0 |\gamma'(\xi)|\, d\xi, \quad \tilde s:= s_\gamma(\tilde u) = \int^{\tilde u}_0 |\gamma'(\xi)| \, d\xi, \quad \text{for} \quad u, \tilde u \in S^1,   
\end{align*}
the solution to \eqref{def-H2-EL} is  
\begin{align*}
\grad\E_\gamma(s) 
 &=\int_0^{\mathcal{L}(\gamma)} G(s,\tilde s;\gamma) \nabla \mathcal E(\gamma)(\tilde s) d\tilde s \\
 &=\int_0^{\mathcal{L}(\gamma)} G [ 2\gamma_{\tilde s \tilde s \tilde s \tilde s}+3(k^2 \gamma_{\tilde s})_{\tilde s} - \lambda^2 \gamma_{\tilde s \tilde s} ] \,d\tilde s \\
 &=\int_0^{\mathcal{L}(\gamma)} [ 2 G_{\tilde s \tilde s \tilde s \tilde s}\gamma - G_{\tilde s} (3k^2 - \lambda^2) \gamma_{\tilde s} ] \,d\tilde s \\
 &=2\gamma(s)+\int_0^{\mathcal{L}(\gamma)} [ 2(G_{\tilde s\tilde s}-G)\gamma -G_{\tilde s} (3k^2 - \lambda^2) \gamma_{\tilde s} ] \,d\tilde s \\
 &=2\gamma(s)-\int_0^{\mathcal{L}(\gamma)} [ 2G\gamma+G_{\tilde s}\gamma_{\tilde s}(3k^2+2-\lambda^2) ]\, d\tilde s,  
\end{align*}
where we have used \eqref{eq:Gde} and the relation $G(s, \tilde s;\gamma)=G(\tilde s, s;\gamma)$. 
Thus we obtain 
\begin{equation}
\label{eq:h2gradu}
\begin{aligned}
\grad\mathcal{E}_\gamma(u)
= 2\gamma(u)-\int_0^1 
\Bigl[ & 2G(s, \tilde s; \gamma) \gamma(\tilde u) |\gamma'(\tilde u)| \\
 & + \frac{1}{|\gamma'(\tilde u)|} \partial_{\tilde u}G(s,\tilde s; \gamma) \gamma'(\tilde u) (3k(\tilde u)^2+2 -\lambda^2) \Bigr] \, d\tilde u. 
\end{aligned} 
\end{equation}
By definition $\grad \E_\gamma\in H^2(S^1,\R^n)$ (this can also be checked directly from \eqref{eq:h2gradu}, 
as in Lemma \ref{superlemma} (iv)) and so problem \eqref{eq:GF} is an ODE in $\I^2(S^1,\R^n)$.

\subsection{Metric completeness}\label{metrics}

The $H^2(ds)$-Riemannian distance between $\gamma,\beta\in \I^2(S^1,\R^n)$ is
\[ \dist(\gamma,\beta):=\inf_{p}\int_0^1 \norm{p'(t)}_{H^2(ds)}\,  dt \]
where the infimum is taken over all piecewise $C^1$ paths $p:[0,1]\to \I^2(S^1,\R^n)$ with $p(0)=\gamma$ and $p(1)=\beta$. 
By Theorem~1.9.5 in \cite{Klingenberg:1995aa}, since $H^2(ds)$ is a strong Riemannian metric\footnote{Note that the definition of a Riemannian metric used in \cite{Klingenberg:1995aa} includes the assumption that it is strong.} 
the distance function defines a metric on $\I^2(S^1,\R^n)$ whose topology coincides with the $H^2$-topology. 

\begin{lemma}[\cite{Bruveris:2015aa}, Lemma 4.2] \label{bruv1}
Write $B^{\dist}_r(\gamma_0)$ for the open ball with radius $r$ with respect to the $H^2(ds)$-Riemannian distance. 
\begin{enumerate}
	\item[{\rm (i)}] Given $\gamma_0\in \I^2(S^1,\R^n)$ there exist $r>0$ and $C>0$ such that 
\[ 
\dist(\gamma_1,\gamma_2)\leq C\htwo{\gamma_1-\gamma_2}
\]
for all $\gamma_1, \gamma_2 \in B^{\dist}_r(\gamma_0)$. 
\item[{\rm (ii)}] Given $B^{\dist}_r(\gamma_0)\subset \I^2(S^1,\R^n)$ there exists $C>0$ such that
\[
\norm{\gamma_1-\gamma_2}_{H^2} \leq C\dist(\gamma_1,\gamma_2)
\]
for all $\gamma_1,\gamma_2\in B^{\dist}_r(\gamma_0)$.
\end{enumerate}
\end{lemma}
\begin{theorem}[\cite{Bruveris:2015aa}, Theorem 4.3] \label{bruv2}
$(\I^2(S^1,\R^n), {\rm dist})$ is a complete metric space. 
\end{theorem}

For contrast once again, the distance obtained from the $L^2(ds)$-metric does not give a complete metric space and in fact vanishes on any path component. 
See~\cite{Michor:2006aa} and the discussion in~\cite{Schrader:2021aa}.

\subsection{Symmetries}
Given a diffeomorphism $\phi\in \text{Diff}(S^1)$ we have the action by reparametrisation $\Phi:\I^2(S^1,\R^n)\to \I^2(S^1,\R^n)$, $\Phi\gamma:=\gamma\circ \phi$. 
Since $\Phi$ is linear we have $d\Phi=\Phi$. 
Along the same lines as in~\cite{Schrader:2021aa} Section~3, this action is an isometry of the $H^2(ds)$-metric. 
That is, 
\[
\ip{\Phi v,\Phi w}_{H^2(ds),\Phi \gamma}=\ip{v,w}_{H^2(ds),\gamma}.  
\]
Moreover, the energy $\E$ is invariant under reparametrisation: $\E(\gamma)=\E(\Phi\gamma)$ and therefore $d\E_\gamma=d\E_{\Phi\gamma}\Phi$. 
Applying the definition of the gradient and the isometry property, we have 
\[
d\E_\gamma v= \ip{\grad\E_\gamma, v}_{H^2(ds),\gamma}= \ip{\Phi \grad\E_\gamma,\Phi v}_{H^2(ds),\Phi\gamma}. 
\]
On the other side $d\E_{\Phi\gamma}\Phi v=\ip{\grad\E_{\Phi\gamma},\Phi v}_{H^2(ds),\Phi\gamma}$ and equating the two we observe that $\Phi \grad \E_\gamma=\grad\E_{\Phi\gamma}$. 
Using the isometry property again we have 
\begin{equation}\label{gradinvariance}
\norm{\grad \E_{\Phi\gamma}}_{H^2(ds)}=\norm{\grad \E_\gamma}_{H^2(ds)}. 
\end{equation} 
By a similar argument, the above also holds when $\Phi$ is the map induced by a fixed translation in $\R^n$.


\section{Existence and uniqueness} \label{section:existence}

From now on, we set 
\begin{equation}
\label{eq:200513}
F(\gamma):= -\grad \mathcal E(\gamma) \quad \text{for} \quad \gamma \in \I^2(S^1,\R^n). 
\end{equation}
Moreover, we denote by $C_S$ the Sobolev constant of the imbedding $H^1(S^1) \subset C^{\frac{1}{2}}(S^1)$: 
\[
\| f \|_{C^{\frac{1}{2}}(S^1)} \le  C_S \| f \|_{H^1(S^1)}. 
\]

\begin{lemma}\label{superlemma}
Let $\gamma_0 \in \I^2(S^1,\R^n)$ and
$
b=\tfrac{1}{2}\min_{u\in S^1}\abs{\gamma'_0(u)}.
$
Then there exist positive constants $c_1,c_2,c_3 $ depending on $\gamma_0$ such that for all $\gamma$ in the open $H^2$-ball $U=B^{H^2}_{b/C_S}(\gamma_0):=\{\gamma\in H^2(S^1,\R^n):\norm{\gamma-\gamma_0}_{H^2}<b/C_s\}$:
\begin{enumerate}
\item[{\rm (i)}] for all $u\in S^1$: $0<c_1<\abs{\gamma'(u)}<c_2$, and therefore $c_1 < \length(\gamma) < c_2$,
\item[{\rm (ii)}] $T:U\to H^1(S^1,\R^n), \gamma \mapsto \gamma_s$ is locally Lipschitz,
\item[{\rm (iii)}] $\kappa:U\to L^2(S^1,\R^n), \gamma \mapsto \gamma_{ss}$ is  locally Lipschitz, and $\norm{k}_{L^2}<c_3$,
\item[{\rm (iv)}] $\norm{F(\gamma)}_{H^2}\leq c_3$,
\item[{\rm (v)}] $\norm{DF_\gamma}_{\op}\leq c_3$. 
\end{enumerate}
where $\norm{\cdot}_{\op}$ is the operator norm.
\end{lemma}
\begin{proof}
By the Sobolev imbedding we have 
\[ 
\bigl| |\gamma'(u)| - |\gamma_0'(u)| \bigr| \leq \| \gamma'-\gamma_0' \|_{C^0} \leq C_S \| \gamma-\gamma_0 \|_{H^2} <b 
\]
for all $u\in S^1$, and then
\[ 
0 < |\gamma_0'(u)| - b < |\gamma'(u)| < b + \|\gamma_0'\|_{C^0} 
\]
which proves (i).

Suppose $\gamma,\mu$ are in the open $H^2$-ball centered at $\gamma_0$ with radius $b/C_S$, then by (i)
\begin{align}\label{eq:Tlip}
|T(\gamma)-T(\mu)| = \abs{\frac{\gamma'}{\abs{\gamma'}}-\frac{\mu'}{\abs{\mu'}}}
&=\frac{1}{|\gamma'| |\mu'|} \abs{ \gamma' |\mu'| - \gamma' |\gamma'| + \gamma' |\gamma'| - \mu' |\gamma'| } \\
&\leq \frac{2}{c_1} |\gamma'-\mu'|
\end{align}
and therefore $\|T(\gamma)-T(\mu)\|_{L^2} \leq \frac{2}{c_1} \|\gamma'-\mu'\|_{L^2}$. 
Showing that 
\[
\|T(\gamma)'-T(\mu)'\|_{L^2}\leq c \|\gamma'-\mu'\|_{H^1}
\]
uses basically the same method but with more terms. 
Thus (ii) follows. 

From $\gamma_{ss}=\gamma'' |\gamma'|^{-2} -\ip{\gamma'',\gamma'} \gamma' |\gamma'|^{-4}$ we get
\[ 
|\kappa|^2=\frac{|\gamma''|^2}{|\gamma'|^4}-\frac{\ip{\gamma'',\gamma'}^2}{|\gamma'|^6}\leq 2 \frac{|\gamma''|^2}{|\gamma'|^4}
\leq \frac{2}{c_1^4} |\gamma''|^2
\]
and then since $|\kappa|=k$:
\[
\int k^2 du=\int |\kappa|^2 du \leq \frac{2}{c_1^4} \|\gamma''\|^2_{L^2} \leq \frac{2}{c_1^4} \Bigl( \frac{b}{C_S} + \|\gamma_0\|_{H^2} \Bigr)^2. 
\]
The proof that 
\[
\| \kappa(\gamma) - \kappa(\mu)\|_{L^2} \leq c \|\gamma-\mu\|_{H^2} 
\]
is similar to (ii). 
Thus (iii) follows. 

From now on the arguments of $G$ and its derivatives, when ommitted, should be taken to be $s,\tilde{s};\gamma$.
By \eqref{eq:h2gradu} and \eqref{eq:200513} we have 
\begin{align} 
\label{eq:F}
F(\gamma)(u)&=-2\gamma(u)+ \int_0^1 
\Bigl[ 2G\gamma(\tilde u) |\gamma'(\tilde u)|\\
& \qquad \qquad \qquad 
+ \frac{1}{|\gamma'(\tilde u)|} \partial_{\tilde u } G \gamma'(\tilde u)\left (3 k(\tilde u)^2+2 - \lambda^2\right )\Bigr]\,d\tilde u, \notag \\
F(\gamma)'(u)&= -2\gamma'(u)+ \int_0^1 \Bigl[ 2 \partial_u G \gamma(\tilde u) |\gamma'(\tilde u)|\\
& \qquad \qquad \qquad 
+ \frac{1}{|\gamma'(\tilde u)|} \partial_u \partial_{\tilde{u}} G \gamma'(\tilde u)\left (3 k(\tilde u)^2+2 - \lambda^2\right )\Bigr]\, d\tilde u,  \label{eq:Fdash} \\
F(\gamma)''(u)&= -2\gamma''(u)+ \int_0^1 \Bigl[ 2 \partial^2_u G \gamma(\tilde u) |\gamma'(\tilde u)|\\
& \qquad \qquad \qquad 
+   \frac{1}{|\gamma'(\tilde u)|} \partial^2_u \partial_{\tilde{u}} G \gamma'(\tilde u)\left (3 k(\tilde u)^2+2-\lambda^2\right )\Bigr]\, d\tilde u.  \label{eq:Fdashdash} 
\end{align}
Then, suppressing arguments in $A(\length(\gamma)-\abs{s-\s},\abs{s-\s})$ and $\beta(\length(\gamma))$, the derivatives of $G$ are 
\begin{align}
\partial_{u} G &=\frac{(\partial_2-\partial_1)A}{\beta}  \sgn (s-\tilde s) |\gamma'(u)|, \label{eq:partialG} \\
\partial_{\tilde u} G &=-\partial_uG, \label{Gu-Gtu} \\
\partial_u \partial_{\tilde u} G &=-\frac{(\partial_2-\partial_1)^2A}{\beta} |\gamma'(\tilde u)| |\gamma'(u)|, \label{Gtuu}\\
\partial^2_u G &=\frac{(\partial_2-\partial_1)^2A}{\beta}  |\gamma'(u)|^2 
            + \partial_{u} G \frac{\ip{\gamma'(u), \gamma''(u)}}{|\gamma'(u)|^2}\label{eq:Guu}, \\
\partial^2_u \partial_{\tilde u} G &= -\frac{(\partial_2-\partial_1)^3 A}{\beta}  \sgn(s-\tilde s) |\gamma'(\tilde u)| |\gamma'(u)|^2 \label{eq:Gu3} \\
& \qquad 
- \frac{(\partial_2-\partial_1)^2A}{\beta}  |\gamma'(\tilde u)| \frac{\ip{\gamma'(u) , \gamma''(u)}}{ |\gamma'(u)|}, \notag
\end{align}
where in the expressions for $\partial_u \partial_{\tilde u}G$ and $\partial^2_{u}G$ we have used the fact, computed from \eqref{eq:A}, that $(\partial_2-\partial_1)A(\length(\gamma),0)$ vanishes. 
Now $A$ and $\beta$ are both smooth functions, and each of $s,\tilde s, \length(\gamma)$ is $L^\infty$-bounded by (i). 
Moreover, $\beta(\length(\gamma))$ is bounded away from zero by (i), and so the quotients in the above expressions are all bounded. 
It follows that $\partial_u G, \partial_{\tilde u} G,$ and $ \partial_u \partial_{\tilde u}G$ are bounded and then from \eqref{eq:F},\eqref{eq:Fdash}, \eqref{eq:partialG}, \eqref{Gu-Gtu} and \eqref{Gtuu}, using (iii), 
we confirm that $\norm{F(\gamma)}_{H^1}$ is bounded. 
From \eqref{eq:Guu},\eqref{eq:Gu3} there exist positive constants $C_1,C_2$ such that
\begin{align*}
|\partial_u^2G(s, \tilde s;\gamma)| & \le C_1(|\gamma'(u)|^2 + |\gamma''(u)|),\\
|\partial_u^2\partial_{\tilde{u}}G(s, \tilde s; \gamma)| & \le C_2 |\gamma'(\tilde u)| ( |\gamma'(u)|^2 + |\gamma''(u)|).
\end{align*}
Using these inequalities with \eqref{eq:Fdashdash}, we have 
\begin{align*}
|F(\gamma)''(u)| &\leq 2|\gamma''(u)|\\
&\quad +\int_0^1\begin{multlined}[t] \Bigl[  2C_1 \bigl(|\gamma'(u)|^2 + |\gamma''(u)| \bigr)|\gamma(\tilde u)| |\gamma'(\tilde{u})|\\
+C_2 |\gamma'(\tilde u)| \bigl(|\gamma'(u)|^2+|\gamma''(u)| \bigr) \bigl( 3k(\tilde u)^2 + 2 + \lambda^2 \bigr) \Bigr] \, d\tilde{u}. 
\end{multlined}
\end{align*}
Using (iii) again we have that $\norm{F(\gamma)''}_{L^2}$ is bounded. 
Thus (iv) follows. 

Let $\alpha:(a,b)\to \I^2(S^1,\R^n)$ be a variation through $\gamma$ in the direction $v\in H^2(S^1,\R^n)$, i.e $\alpha(0)=\gamma$ and $\alpha'(0)=v$. 
Treating $\alpha$ as a function of two variables $\alpha(\varepsilon,u)$, we calculate 
\[ 
\partial_\varepsilon |\alpha_u| \Bigm|_{\varepsilon=0}=\ip{\alpha_{u \varepsilon},\alpha_u} |\alpha_u|^{-1}\Bigm|_{\varepsilon=0}
=\ip{\alpha_{\varepsilon u},\alpha_s} \Bigm|_{\varepsilon=0} = \ip{ v_u,\gamma_s }.  
\]
Hence also at $\varepsilon=0$ we have $\partial_\varepsilon ds=\langle v_s,\gamma_s \rangle ds$. 
Similarly for any function $\phi$ depending on $\alpha$ we find the following  relation for commutation with the arc length derivative
\[ 
\phi_{s\varepsilon}\Bigm|_{\varepsilon=0}=\phi_{\varepsilon s}-\ip{\alpha_{\varepsilon s},\alpha_s}\phi_s \Bigm|_{\varepsilon=0} 
 = \phi_{\varepsilon s}-\langle v_s,\gamma_s\rangle \phi_s. 
\]
Using these formulae we calculate
\begin{equation}
\label{DF}
\begin{alignedat}{2}
DF_\gamma v &=\partial_\varepsilon F(\alpha)\bigm|_{\varepsilon=0} \\
& =-2v + \int_0^{\length(\gamma)} \Bigl[ 2DG_\gamma v+2Gv +2G\gamma \ip{v_{\tilde s},\gamma_{\tilde s}} \\
& \qquad \qquad \qquad \quad \,\, + \bigl( (DG_\gamma v)_{\tilde s}\gamma_{\tilde s}+G_{\s}v_{\tilde s} \bigr) (3k^2+2-\lambda^2) \\
& \qquad \qquad \qquad \quad \,\, + G_\s \gamma_\s \bigl( 6\ip{v_{\s \s },\gamma_{\s \s }} - (15k^2 + 2 - \lambda^2) \ip{v_\s ,\gamma_\s } \bigr) \Bigr] \, d\s.  
\end{alignedat}
\end{equation}
To see that $\norm{DF_\gamma}_{\op}$ is bounded, note  from \eqref{eq:defG} that $G(x,y)$ and $\partial_2G(x,y)$ are both continuous functions,  and so by (i) we have $L^\infty$-bounds for $G(s,\tilde s)$ and $G_{\s}(s,\tilde s)$. 
Since (iii) gives an $L^2$-bound for $k$, it remains to show that $DG_\gamma v$ and $(DG_\gamma v)_s$ are bounded. 
From $s_\gamma=\int_0^u |\gamma'(\tau)|\,d\tau$ and $\length(\gamma)=s_\gamma(1)$ we calculate 
\begin{align*}
Ds_\gamma v&=\int_0^u \ip{v',\gamma'/|\gamma '|}\, d\tau, \\
D\length_\gamma v &=\int_0^1 \ip{v', \gamma'/|\gamma '|}\, d\tau, \\
D(\abs{s-\tilde s})_\gamma v &=\sgn(s-\s)\int_{\tilde u}^u\ip{v',\gamma'/|\gamma'|}\, d\tau, 
\end{align*}
and, assuming $\norm{v}_{H^2}=1$, each of these is $L^\infty$-bounded. Then, suppressing arguments in $A$ and $\beta$ again we have
\[
DG_\gamma v = \frac{1}{\beta}\partial_1 A\left (D\length_\gamma v -D(\abs{s-\s})_\gamma v\right )+ \frac{1}{\beta}\partial_2 A D(\abs{s-\tilde s})_\gamma v-\frac{\beta'}{\beta^2}A D\length_\gamma v.
\]
Since $A$ and $\beta$ are smooth functions, and $\length(\gamma)$ is bounded away from zero by (i), we now have that $|DG_\gamma|$ is bounded. Next we calculate 
\begin{align*}
(DG_\gamma v)_s &= \begin{multlined}[t]\frac{1}{\beta} (\partial_2\partial_1 A-\partial_1^2 A )\sgn(s-\s)(D\length_\gamma v -D(\abs{s-\s})_\gamma v )) \\
+ \frac{1}{\beta}(\partial_2 A - \partial_1 A)(D(\abs{s-\s})_\gamma v)_s \\
  +\frac{1}{\beta} (\partial_2^2A-\partial_1\partial_2 A)\sgn(s-\s)D(\abs{s-\s})_\gamma v \\
-\frac{\beta'}{\beta^2}(\partial_2 A-\partial_1 A)\sgn(s-\s) D\length_\gamma v, 
\end{multlined}\\
(D(\abs{s-\tilde s})_\gamma v)_s &= \sgn(s-\s)v_s\gamma_s.
\end{align*}
This implies that $(DG_\gamma v)_s $ is also $L^\infty$-bounded when $\norm{v}_{H^2}=1$, and then we observe from \eqref{DF} that $\|DF_\gamma \|_{\op}$ is bounded. 
Thus (v) follows, and the proof of Lemma \ref{superlemma} is complete. 
\end{proof}

\begin{proposition}\label{shortexistence}
Let $\gamma_0 \in \I^2(S^1,\R^n)$. 
Then there exists $T>0$ such that problem \eqref{eq:GF} possesses a unique solution in $C^1([0, T), \I^2(S^1, \mathbb{R}^n))$. 
\end{proposition}
\begin{proof}
According to the Banach space Picard--Lindel\"of theorem in \cite[Theorem~3.A]{Zeidler:1986aa} it is sufficient to show that there is an $H^2$-ball containing $\gamma_0$ on which $F$ is bounded and Lipschitz. Using the ball from Lemma \ref{superlemma}, $F$ is bounded by (iv) and is Lipschitz by (v) and the mean value inequality (e.g. see \cite[Corollary. 4.2]{Lang:1999sf}). 
\end{proof}

For long-time existence we will imitate the method used in e.g. \cite[Theorem 9.1.6]{Palais:1988fv}.
\begin{lemma}\label{finitelengthlimit}
Assume that $\gamma:(a,b)\to \I^2(S^1,\R^n)$ is a $C^1$ curve with finite $H^2(ds)$-length. 
Then $\lim_{t\to b}\gamma(t)$ exists in $(\I^2(S^1,\R^n), {\rm dist})$.
\end{lemma}
\begin{proof}
For \emph{any} increasing sequence of times $t_i\to b$ we claim that $\gamma_i:=\gamma(t_i)$ is Cauchy in $(\I^2(S^1,\R^n), {\rm dist})$. 
Suppose not, then there exists $\varepsilon>0$ such that for all $N>0$ there exist $j,k >N$ such that $\dist(\gamma_j,\gamma_k)>\varepsilon$. 
Since $\dist(\gamma_j,\gamma_k)<\int_{t_j}^{t_k}\norm{\gamma_t}_{H^2(ds)}\, dt$, this contradicts the assumption that the length $\int_a^b \norm{\gamma_t}_{H^2(ds)}\, dt$ should be finite. 
By Theorem \ref{bruv2} we see that $\gamma_i$ converges to some $\gamma_b$. 
Moreover, $\gamma_b$ is unique: 
if we have another sequence $\tilde \gamma_i:=\gamma(\tilde t_i)\to \tilde \gamma_b$ and we take $\bar t_i$ to be the ordered union of $t_i$ and $\tilde t_i$ then $\gamma(\bar t_i)$ must also converge, 
which requires $\tilde \gamma_b = \gamma_b$. Finally, we have $\lim_{t\to b}\gamma(t)=\gamma_b$, because if not then there exists an increasing sequence $t_j\to b$ such that $\gamma(t_j)$ does not converge to $\gamma_b$.
\end{proof}

\begin{proposition} \label{git-existence}
Assume that $\gamma_0\in \I^2(S^1,\R^n)$. 
Then problem \eqref{eq:GF} possesses a unique global-in-time solution $\gamma$. 
\end{proposition}
\begin{proof}
Suppose $\gamma:[0,T)\to \I^2(S^1,\R^n)$ is a maximal solution curve to $\gamma_t=-\grad \E_\gamma$. 
Abbreviating $\E(t):=\E(\gamma(t))$, since 
\[ 
\E(t)-\E(0)=\int_0^t \E'(\tau) \, d\tau = -\int_0^t \|\grad \E_\gamma \|_{H^2(ds)}^2 \, d\tau 
\]
we have 
\begin{equation}\label{l2time}
\int_0^T \|\grad \E_\gamma\|_{H^2(ds)}^2 \, dt \leq \E(0), 
\end{equation}
and then using the H\"older inequality we obtain  
\[ 
\int_0^T \|\grad \E_\gamma\|_{H^2(ds)}\, dt \leq \sqrt{T} \Bigl(\int_0^T \|\grad \E_\gamma\|_{H^2(ds)}^2 \, dt \Bigr)^{1/2} \leq \sqrt{T\E(0)}. 
\]
Observe that if $T$ is finite then the length of $\gamma$ is finite, and then it follows from Lemma \ref{finitelengthlimit} that $\lim_{t\to T}\gamma(t)$ exists in $(\I^2(S^1,\R^n), {\rm dist})$. 
This contradicts the maximality of $T$. 
Therefore Proposition \ref{git-existence} follows. 
\end{proof}


\section{\L ojasiewicz--Simon gradient inequality } \label{section:LSinequality}

Our general strategy for proving convergence of the flow to stationary points is well known: 
upgrade the $L^2$-in-time estimate \eqref{l2time} to an $L^1$ estimate using a \L ojasiewicz--Simon gradient inequality, and then the length of any trajectory is finite (see e.g. Chapter 7 of \cite{Feehan:2016aa} for abstract results). Typically the \L ojasiewicz--Simon inequality requires that $\E$ should be analytic, which we verify in Section \ref{subsection:analyticity}. The other standard requisite is a Fredholm second derivative, and so as mentioned in the introduction, an immediate obstruction is that the reparametrisation invariance of $\E$ implies that its second derivative has infinite dimensional nullspace, and is therefore not Fredholm. 
To get around this, we restrict $\E$ to the submanifold of arc length proportionally parametrized curves. 
We show that on this submanifold a gradient inequality is satisfied, and then use the reparametrisation symmetry to extend the inequality to the full space.

\subsection{Regularity of critical points and the second variation}
First we show that stationary points of $\E$ have higher regularity, and then we calculate the second derivative of $\E$. The higher regularity of critical points will be required in order to show in Proposition \ref{fredholm} that the second derivative is Fredholm.
\begin{proposition}\label{critsmooth}
$\gamma\in \htwospace $ is a stationary point of $\E$ iff $\partial_s^4\gamma\in L^2$ and 
\begin{align}
\label{EL-k}
& 2 \partial^4_s \gamma + \partial_s(3k^2\gamma_s-\lambda^2\gamma_s)=0\quad \text{a.e.}, \\ 
&\partial^2_s \gamma(0)=\partial^2_s \gamma(\length(\gamma)), \quad \partial^3_s \gamma(0)=\partial^3_s \gamma(\length(\gamma)). \label{eq:per}
\end{align}
\end{proposition}
\begin{proof}
Recalling \eqref{dE} we have 
\begin{equation}\label{dE2}
d \E_\gamma V =\int^{\length}_0 \bigl[ 2\ip{\gamma_{ss},V_{ss}}-(3k^2-\lambda^2)\ip{\gamma_s,V_s} \bigr] \,ds
,
\end{equation} 
where $\length:=\length(\gamma)$. 
Assume $\gamma$ is a stationary point of $\E$, i.e. $d\E_{\gamma} V=0$ for all $V\in \htwospace$.
Let $y$ and $w$ be the solutions to 
\begin{equation}\label{eq:defy}
y_s=3k^2\gamma_s-\lambda^2\gamma_s,\qquad y(0)=0,
\end{equation}
\[ w_{ss}=2\gamma_{ss}+y,\qquad w(0)=0, \quad w_s(0)=0, \]
and define 
\[ 
\tilde{w}(s):= w(s)-\frac{s^2}{\length^2} w(\length)-\frac{s^2}{\length^2}(s-\length) \Bigl( w_s(\length)-\frac{2}{\length}w(\length) \Bigr), 
\]
so that $\tilde w(0),\tilde w(\length),\tilde w_s(0)$ and $\tilde w_s(\length)$ all vanish and $\tilde w \in \htwospace$. 
Moreover, $\tilde w_{ss}=w_{ss}-P$ where $P_{ss}=0$ and so $\int_0^\length \ip{P,\tilde w_{ss}} ds=0$. 
Since $\gamma$ is stationary $d\E_\gamma \tilde{w}=0$, from which it follows that
\begin{align*}
0=\int_0^\length \ip{2\gamma_{ss}+y,\tilde{w}_{ss}}\, ds 
&=\int^\length_0 \ip{2\gamma_{ss}+y-P,\tilde w_{ss}}\, ds\\
&=\int_0^\length | 2\gamma_{ss}+y-P |^2 \, ds
\end{align*}
and therefore 
\begin{equation}\label{eq:weakELE}
2\gamma_{ss}+y-P=0 \quad \text{a.e. in} \quad S^1. 
\end{equation} 
Considering only the normal component, since $P$ is smooth and from \eqref{eq:defy} $y$ is $H^1$, we observe that $k$ is also in $H^1$. 
But then $y_{uu}=6kk_u\gamma_u+(3k^2-\lambda^2)\gamma_{uu}$ is in $L^2$, so $y$ is $H^2$ and so is $k$. 
Note that since \eqref{eq:weakELE} gives no information about the \emph{tangential} component of $\gamma_{uu}$ 
we cannot bootstrap any further\footnote{$y_{uuu}=(6k_u^2+6kk_{uu})\gamma_u+12kk_u\gamma_{uu}+(3k^2-\lambda^2)\gamma_{uuu}$, but $\gamma_{uuu}$ is not necessarily $L^2$.}.
However, we already have enough $L^2$-derivatives to differentiate \eqref{eq:weakELE} twice with respect to $s$ and for $\gamma$ to satisfy the Euler-Lagrange equation \eqref{EL-k} (a.e.).
As for the periodicity conditions \eqref{eq:per}, integration by parts in \eqref{dE2} gives
\begin{align*}
d \E_\gamma V =	\begin{multlined}[t] 2\int^\length_{0} \ip{\gamma_{s^4}+((3k^2-\lambda^2)\gamma_s)_s,V}\, ds \\
 + \Bigl[ 2\ip{\gamma_{ss},V_s} -2\ip{\gamma_{s^3},V} +\ip{(3k^2-\lambda)\gamma_{s},V}\Bigr]_{0}^{\length}. 
\end{multlined}
\end{align*}
For $\gamma$ to be stationary the boundary terms must vanish for all $V\in H^2(S^1,\R^n)$. Choosing $V$ such that $V(0)=0=V(\length)$ we see that $\gamma_{ss}$ (and therefore $k$) must be periodic. Then, setting $V_s(0)=V_s(\length)$ (but $V(0)\neq 0$) and using the periodicity of $k$ and $\gamma_s$ we confirm that $\gamma_{s^3}$ is periodic. 
\end{proof}

Next we calculate the second variation of $\E$: 
\begin{lemma} \label{L-second-var}
Let $\gamma \in \I^2(S^1,\R^n)$. Then for $V,W\in H^2(S^1,\R^n)$ the second variation formula at $\gamma$ is given by 
\begin{equation}
\begin{aligned}\label{d2E}
d^2\E_\gamma(V,W) 
= \int_0^{\length(\gamma)} \bigl[ & \ip{W_{ss}, 2V_{ss}-2\ip{V_{ss},T}T-2\ip{V_s,\kappa}T-6\ip{V_s,T}\kappa } \\
& +\ip{W_s,  -2\ip{V_s,\kappa}\kappa -6\ip{V_{ss},\kappa}T-2\ip{V_{ss},T}\kappa }\\
& +\ip{W_s,-(3k^2-\lambda^2)V_s+(15k^2-\lambda^2)\ip{V_s,T}T} \bigr]\, ds. 
\end{aligned} 
\end{equation}
\end{lemma}
\begin{proof}
Consider $d\E: U \to \htwospace^*$, $U$ an open subset of $\I^2(S^1,\R^n)$, and let $\alpha:=\gamma+\varepsilon W$ be a variation of $\gamma\in U$ in the direction of $W\in \htwospace$. 
Then using $\partial_\varepsilon ds=\ip{\alpha_{\varepsilon s}, \alpha_s}ds$, from \eqref{dE} we calculate
\begin{align*}
\partial_\varepsilon d\E_\alpha V 
=\begin{multlined}[t] \int_0^\length \bigl[ 2\ip{V_{ss\varepsilon},\alpha_{ss} }+2\ip{V_{ss},\alpha_{ss\varepsilon} }-6\ip{\alpha_{ss\varepsilon},\alpha_{ss}}\ip{V_s,\alpha_s}\\	
-(3k^2-\lambda^2) \bigl( \ip{V_{s\varepsilon},\alpha_s}+\ip{V_s,\alpha_{s\varepsilon}} \bigr) \\
+\bigl( 2\ip{V_{ss},\alpha_{ss}}-(3k^2-\lambda^2)\ip{V_s,\alpha_s} \bigr)\ip{\alpha_{\varepsilon s},\alpha_s} \bigr] \,ds. 
\end{multlined}
\end{align*}
Using the formula for commutation with the arc length derivative:
\[ 
\phi_{s\varepsilon}=\phi_{\varepsilon s}-\ip{\alpha_{\varepsilon s},\alpha_s}\phi_s, 
\]
and writing $T=\gamma_s$, we calculate
\begin{align*}
\alpha_{s\varepsilon}\bigm|_{\varepsilon=0}&=\alpha_{\varepsilon s}-\ip{\alpha_{\varepsilon s},\alpha_s}\alpha_s \bigm|_{\varepsilon=0}= W_s-\ip{W_s,T}T, \\
\alpha_{ss\varepsilon}\bigm|_{\varepsilon=0} 
&=\alpha_{\varepsilon s s}-\ip{\alpha_{\varepsilon ss},\alpha_s}\alpha_s -\ip{\alpha_{\varepsilon s},\alpha_{ss}}\alpha_s-2\ip{\alpha_{\varepsilon s},\alpha_s}\alpha_{ss} \bigm|_{\varepsilon=0}\\
&= W_{ss}-\ip{W_{ss},T}T-\ip{W_s,\kappa}T-2\ip{W_s,T}\kappa. 
\end{align*}
Moreover, noting that $V$ is independent of $\varepsilon$, we have 
\begin{align*}
V_{s\varepsilon}\bigm|_{\varepsilon=0}&= -\ip{\alpha_{\varepsilon s},\alpha_s}V_s \bigm|_{\varepsilon=0} = -\ip{W_s,T}V_s, \\
V_{ss\varepsilon}\bigm|_{\varepsilon=0}
&=-\ip{\alpha_{\varepsilon ss}, \alpha_s}V_s-\ip{\alpha_{\varepsilon s},\alpha_{ss}}V_s-2\ip{\alpha_{\varepsilon s},\alpha_s}V_{ss} \bigm|_{\varepsilon=0}\\
&= -\ip{W_{ss},T}V_s-\ip{W_s,\kappa}V_s-2\ip{W_s,T}V_{ss}.
\end{align*}
Therefore, evaluating $\partial_\varepsilon d\E_\alpha V$ at $\varepsilon =0$ and using $\ip{\alpha_{ss},\alpha_s=0}$:
\begin{equation*}
\begin{aligned}
&d^2\E_\gamma (V,W) \\
&= \int_0^{\length(\gamma)} \bigl[ -2\ip{W_{ss},T}\ip{V_s,\kappa}	-2\ip{W_s,\kappa}\ip{V_s,\kappa}-4\ip{W_s,T}\ip{V_{ss},\kappa} \\
& \qquad \qquad \quad +2\ip{V_{ss},W_{ss}}-2\ip{W_{ss},T}\ip{T,V_{ss}}-2\ip{W_s,\kappa}\ip{T,V_{ss}} \\
& \qquad \qquad \quad -4\ip{W_s,T}\ip{\kappa,V_{ss}} -6\ip{W_{ss},\kappa}\ip{V_s,T}+12k^2\ip{W_s,T}\ip{V_s,T}\\
& \qquad \qquad \quad +(3k^2-\lambda^2)\left (2\ip{W_s,T}\ip{V_s,T}-\ip{V_s,W_s}\right)\\
& \qquad \qquad \quad +(2\ip{V_{ss},\kappa}-(3k^2-\lambda^2)\ip{V_s,T})\ip{W_s,T} \bigr]\, ds \\ 
&=
\int_0^{\length(\gamma)} \bigl[ \ip{W_{ss}, 2V_{ss}-2\ip{V_{ss},T}T-2\ip{V_s,\kappa}T-6\ip{V_s,T}\kappa } \\
& \qquad \qquad \quad +\ip{W_s,  -2\ip{V_s,\kappa}\kappa -6\ip{V_{ss},\kappa}T-2\ip{V_{ss},T}\kappa }\\
& \qquad \qquad \quad +\ip{W_s,-(3k^2-\lambda^2)V_s+(15k^2-\lambda^2)\ip{V_s,T}T} \bigr]\, ds 
\end{aligned}
\end{equation*}
which proves the lemma. 
\end{proof}

\subsection{Analyticity} \label{subsection:analyticity}
We are assuming less differentiability than \cite[Section 3.1]{DallAcqua:2016aa} and so the proof is a little different. 
\begin{definition}[cf. \cite{DallAcqua:2016aa}, Section 3.1]	
Let $X, Y$ be Banach spaces and $D\subset X$ an open subset. 
A map $f:D\to Y$ is \emph{analytic at $x_0\in D$} if there exist continuous symmetric multilinear functions $a_j:X^j=X\times \ldots \times X\to Y,\, j\in \mathbb{N} $ and $a_0\in Y$ such that in a neighbourhood of $x_0$
\[
\sum_{j=0}^\infty \norm{a_j}_{\op} \norm{x-x_0}^j <\infty, \quad  \text{and} \quad f(x)=\sum_{j=0}^\infty a_j(x-x_0)^j,
\]
where $\norm{a_j}_{\op} $ is the operator norm and $a_j(x-x_0)^j:=a_j(x-x_0,\ldots ,x-x_0)$.
\end{definition}
Our goal is to prove that $\E:H^2(S^1,\R^n)\to \R$ is analytic. The proof will require a series of lemmas showing that the constituent functions such as $\gamma \mapsto |\gamma_u|$ are analytic. 
Let $H^1_{+}(S^1,\R):=\{ \alpha \in H^1(S^1,\R ): \alpha(u)>0 \}$ and define 
\[ f:H^1_{+}(S^1,\R)\to H^1_{+}(S^1,\R),\quad f(\alpha)(u):=\sqrt{\alpha(u)}.\] 
Note that by the Sobolev imbedding $\alpha$ is continuous and so 
\[ 
\norm{f(\alpha)}_{H^1}^2=\int_{S^1} \alpha(u) \,du +\int_{S^1}\frac{1}{4\alpha}\alpha_u^2\, du \leq \norm{\alpha}_{L^\infty} + \frac{1}{4}(\min |\alpha|)^{-1}\norm{\alpha_u}^2_{L^2}<\infty
\]
which shows that $f(\alpha)$ is indeed in $H^1$. 
\begin{lemma}\label{sqrt}
The function $f:H^1_{+}(S^1,\R)\to H^1_{+}(S^1,\R)$ is analytic.
\end{lemma}
\begin{proof}
Let $\alpha_0\in H^1_{+}(S^1,\R)$ and consider the open $H^1$-ball $B_\varepsilon(\alpha_0)$ with radius $\varepsilon$ centred at $\alpha_0$. 
Assume $\varepsilon C_S<\min_{u \in S^1}|\alpha_0(u)|$, where $C_S$ is the Sobolev imbedding constant, so that for all $\alpha\in B_\varepsilon(\alpha_0)$ we have
\begin{equation}\label{ball}
0<\alpha_0(u)-\varepsilon C_S < \alpha(u) <  \varepsilon C_S+\alpha_0(u) 
\end{equation}
for all $u \in S^1$, and therefore $B_\varepsilon(\alpha_0)\subset H^1_{+}(S^1,\R)$. Considered as a function from $\R^+\to \R^+$, the square root is analytic. Indeed, the Taylor series about $x_0$ is
\[ 
\sum_{j=0}^\infty b_j(x_0)(x-x_0)^j, \quad b_j(x_0):=c_j x_0^{\frac{1}{2}-j}, \quad c_j:=\frac{(-1)^{j-1}(2j-3)!}{2^{2j-2}j!(j-2)!}, 
\]
the ratio of successive terms is $\frac{(1-2j)}{2(j+1)} (x-x_0)x_0^{-1}$, 
and so by the ratio test the series converges when $|x-x_0|< x_0$. It follows from \eqref{ball} that if $\alpha\in B_\varepsilon(\alpha_0)$ then for each $u$ the series 
\[ 
\sum_{j=0}^\infty b_j(\alpha_0(u))(\alpha(u)-\alpha_0(u))^j
\]
converges to $f(\alpha)(u)$. 
Moreover, since $\alpha_0(u)>0$ we have 
\[
\norm{b_j(\alpha_0)}_{H^1}^2 \leq c_j^2 \Bigl(\min_{u \in S^1} \alpha_0 \Bigr)^{1-2j} + \Bigl(\frac{1}{2}-j \Bigr)^2 c_j^2 \Bigl(\min_{u \in S^1} \alpha_0 \Bigr)^{-1-2j} \norm{\partial_u \alpha_0}_{L^2}^2, 
\]
and so $b_j(\alpha_0)\in H^1$. 
Since $H^1(S^1,\R)$ is a Banach algebra (see e.g. \cite[Theorem 4.39]{Adams:2003aa}) we may identify $b_j(\alpha_0)$ with the linear map from $H^1(S^1,\R)$ to itself given by $\alpha(u)\mapsto b_j(\alpha_0(u))\alpha(u) $.
Then to conclude the proof we need to show that $\sum_{j=0}^\infty \norm{ b_j(\alpha_0)}_{\op}\norm{\alpha-\alpha_0}_{H^1}^j$ converges, where the norm on $b_j(\alpha_0)$ is the operator norm. 
Note that
\[ b_{j+1}(\alpha_0)=\frac{1-2j}{2j+2}\alpha_0^{-1}b_j(\alpha_0). \]
Then using the Banach algebra property (\cite[Theorem 4.39]{Adams:2003aa}) there is a constant $K$ such that 
\[ 
\|b_{j+1}(\alpha_0)\|_{\op}=\sup_{\|v\|_{H^1}=1} \Bigl\|\frac{1-2j}{2j+2}\alpha_0^{-1}b_j(\alpha_0)v \Bigr\|_{H^1} 
\leq K \Bigl\| \frac{1-2j}{2j+2}\alpha_0^{-1} \Bigr\|_{H^1} \|b_j(\alpha_0)\|_{\op} 
\]
and so the ratio of successive terms is
\begin{multline*}
\frac{\|b_{j+1}(\alpha_0)\|_{\op}}{\|b_j(\alpha_0)\|_{\op}} \|\alpha-\alpha_0\|_{H^1} 
\leq K \Bigl\|\frac{1-2j}{2j+2}\alpha_0^{-1}\Bigr\|_{H^1} \|\alpha-\alpha_0\|_{H^1} \\
\to K \|\alpha_0^{-1}\|_{H^1} \|\alpha-\alpha_0\|_{H^1}	
\end{multline*}
as $j \to \infty$. 
By the ratio test the series converges if $\|\alpha-\alpha_0\|_{H^1} < K^{-1} \|\alpha_0^{-1}\|_{H^1}^{-1}$ which can be arranged by choosing $\varepsilon$ sufficiently small. 
\end{proof}

\begin{lemma}\label{speed}
The function $\I^2(S^1,\R^n) \to H^1_{+}(S^1,\R),\gamma \mapsto |\gamma_u|$ is analytic.
\end{lemma}
\begin{proof}
First note that $\partial_u: \I^2(S^1,\R^n) \to H^1(S^1,\R^n)$ is linear and continuous, therefore analytic. 
The map induced by the Euclidean inner product $\ip{\cdot,\cdot} : H^1(S^1,\R^n) \times H^1(S^1,\R^n) \to H^1(S^1,\R)$ satisfies 
\[
\| \ip{v,w} \|_{H^1}^2 = \int_{S^1} \bigl[ |\ip{v,w}|^2 + |\ip{v',w} + \ip{v,w'}|^2 \bigr] \, du \le c \|v\|_{H^1}^2 \|w\|_{H^1}^2. 
\]
So it is continuous bilinear and therefore the map $q(\gamma):= \ip{\gamma,\gamma}$ is analytic. 
Now by Lemma \ref{sqrt} we have that $\gamma \mapsto |\gamma_u|$ is a composition of analytic functions $f \circ q\circ \partial_u$ 
and is therefore analytic by e.g. \cite[p. 1079]{Whittlesey:1965aa}. 
\end{proof}

\begin{lemma}
The function $\rec: H^1_{+}(S^1,\R)\to H^1_{+}(S^1,\R),\, \psi(u)\mapsto \psi(u)^{-1}$ is analytic.
\end{lemma}
\begin{proof}
This is proved for $H^3$ functions in \cite[Appendix B.1]{DallAcqua:2016aa}, and the same method of proof works for $H^1$ functions.
\end{proof}

We also require the following property of analytic functions: 
If $F:D\to Y_1, \, G:D\to Y_2$ are analytic and there exists a bilinear continuous mapping $*:Y_1\times Y_2\to Z$ into another Banach space, 
then the product $F*G:D\to Z, x\mapsto F(x)*G(x)$ is also analytic. 
According to \cite[p. 2175]{DallAcqua:2016aa} this can be proved using similar ideas as for the Cauchy product of series.
\begin{lemma}\label{diffsanalytic}
The function $\partial_s:\I^2(S^1,\R^n) \to H^1(S^1,\R^n)$, $\gamma \mapsto \gamma_s= \gamma_u / |\gamma_u|$ is analytic. 
\end{lemma}
\begin{proof}
We observed in the two previous lemmas that $\partial_u$ and $\gamma \mapsto |\gamma_u|^{-1}$ are analytic on $\I^2(S^1,\R^n)$, with respective codomains $H^1(S^1,\R^n)$ and $H^1_+(S^1,\R)$. Moreover the pointwise product $*:H^1(S^1,\R)\times H^1(S^1,\R^n)\to H^1(S^1,\R^n)$ is bilinear and continuous. Indeed if $\alpha \in  H^1(S^1,\R)$ and $\gamma \in H^1(S^1,\R^n)$ then
\[
\|\alpha*\gamma\|_{H^1}\leq \|\alpha\|_{L^\infty} \|\gamma\|_{L^2} + \|\alpha\|_{L^\infty} \|\gamma_u\|_{L^2} + \|\alpha_u\|_{L^2} \|\gamma\|_{L^\infty}. 
\]
Since $\partial_s\gamma$ is the pointwise product of $\gamma_u$ and $|\gamma_u|^{-1}$, the result follows. 
\end{proof}

\begin{lemma}\label{diffss}
The function $\partial_s^2:\I^2(S^1,\R^n) \to L^2(S^1,\R^n)$ is analytic.
\end{lemma}
\begin{proof}
Since $\partial_u:H^1(S^1,\R^n)\to L^2(S^1,\R^n)$ is also linear and continuous, 
by Lemma \ref{diffsanalytic} we have that $\partial_u\partial_s:\I^2(S^1,\R^n) \to L^2(S^1,\R^n)$ is a composition of analytic functions, therefore analytic. 
The product $*:H^1(S^1,\R)\times L^2(S^1,\R^n)\to L^2(S^1,\R^n)$ is also bilinear and continuous, and $\gamma_{ss}=|\gamma_u|^{-1} \gamma_{su}$, so $\partial_s^2$ is analytic.
\end{proof}

\begin{proposition} \label{energy-analytic-1}
The energy $\E: \I^2(S^1,\R^n) \to \R$ is analytic. 
\end{proposition}
\begin{proof}
Recall that $\E(\gamma)=\int_{S^1} [ \ip{\gamma_{ss},\gamma_{ss}}|\gamma_u|+\lambda |\gamma_u|] \, du$. 
Lemma \ref{diffss} implies that $\partial_s^2$ is analytic. 
Moreover, the Euclidean inner product is continuous bilinear on $L^2(S^1,\R^n)\times L^2(S^1,\R^n)\to L^1(S^1,\R)$, 
pointwise multiplication $*: L^1(S^1,\R)\times H^1(S^1,\R) \to L^1(S^1,\R)$ is continuous bilinear, and the sum of analytic functions is analytic. 
Thus the integrand is analytic as a function $\I^2(S^1,\R^n) \to L^1(S^1,\R)$. 
Integration is of course linear and bounded on $L^1$, so the energy is a composition of analytic functions.
\end{proof}

\subsection{The submanifold of arc length proportional parametrized curves} \label{subsection:al-para-curves}
Denote $H^1_{zm}(S^1,\R):=\{\alpha\in H^1(S^1,\R): \int_{S^1} \alpha \,du=0\}$ and define 
\[ \Phi:\I^2(S^1,\R^n) \to H^1_{zm}(S^1,\R),\quad \Phi(\gamma):=|\gamma_u|-\length(\gamma).\]
Then $\Omega:=\Phi^{-1}(0)$ is the subset of $\I^2(S^1,\R^n)$ consisting of curves which are parametrized proportional to arc length. 
\begin{proposition} \label{proposition:Omega-210630}
The set $\Omega$ of arc length proportional parametrized curves is an analytic submanifold of $H^2(S^1,\R^n)$. 
\end{proposition}

\begin{proof}
We will show that $\mathbf{0}\in H^1(S^1,\R)$ is a regular value of $\Phi$. For the derivative we get
\begin{equation}\label{dphi}
d\Phi_\gamma v=\frac{1}{|\gamma'|}\ip{v', \gamma' }-\int_{S^1} \frac{1}{|\gamma'|}\ip{v', \gamma' }du,   
\end{equation}
and so $\mathbf 0$ is a regular value if for all $w \in H^1_{zm}(S^1,\R)$ there exists $v\in H^2(S^1,\R^n)$ such that 
\begin{align}\label{surjective}
\frac{1}{|\gamma'|}\ip{v', \gamma' }-\int_{S^1} \frac{1}{|\gamma'|}\ip{v', \gamma' }du =w. 
\end{align}
To show that such a $v$ exists we will need
an orthonormal frame $\{T,\nu_i\}$ along $\gamma$ which we construct as follows. First let $\{T(0),\nu_i^0\}$ be an orthonormal basis at $\gamma(0)$, and then let $\nu_i(u), \,i=1,\ldots n-1$ be the solutions to 
\begin{equation}\label{frame}
\nu_i'=-\frac{1}{\abs{\gamma'}^2}\ip{\nu_i,\gamma''}\gamma', \quad \nu_i(0)=\nu_i^0 \, .
\end{equation}
Then 
\begin{align*}
	\frac{d}{du}\ip{\gamma',\nu_i}&=\ip{\gamma'',\nu_i}-\frac{1}{\abs{\gamma'}^2}\ip{\gamma',\gamma'}\ip{\nu_i,\gamma''}=0
\end{align*}
and since $\ip{T(0),\nu_i^0}=0$ we have $\ip{\gamma',\nu_i}=0$ identically.  Moreover
\begin{align*}
	\frac{d}{du}\ip{\nu_i,\nu_j}&=-\frac{1}{\abs{\gamma'}^2}\ip{\nu_i,\gamma''}\ip{\gamma',\nu_j}-\frac{1}{\abs{\gamma'}^2}\ip{\nu_j,\gamma''}\ip{\nu_i,\gamma'}=0
\end{align*}
from which it follows that the orthonormality of $\{T(0),\nu_i^0\}$ is indeed preserved and $\{T,\nu_i\}$ is an orthonormal frame.

Now given $w\in H^1_{zm}(S^1,\R)$ we let $v$ be a solution of 
\begin{equation}\label{controlsystem}
v'=w T+\beta\sum_{i=1}^{n-1}\nu_i \xi_i,  
\end{equation}
where $\beta:S^1\to \R$ is any smooth function satisfying $\beta(0)=0=\beta(1)$ and thereby ensuring $v'(0)=v'(1)$, and $\xi\in H^1(S^1, \R^{n-1})$ is a \emph{control} function which we are free to choose in order to ensure $v$ satisfies the zeroth order periodicity condition $v(0)=v(1)$.
If this is possible then the solution $v$ satisfies \eqref{surjective} and $d\Phi_\gamma$ is surjective. 

In fact, it will be sufficient to consider the system 
\begin{equation}\label{cs2}
x'=\beta \sum_{i=1}^{n-1}\nu_i \xi_i
\end{equation}
because if $y$ is any solution to $y'=w T$ and we can control $x$ from e.g. $x(0)=y(0)$ to $x(1)=y(1)$, then we let $v=y-x$ and $v(0)=0=v(1)$.
According to \cite[p. 76]{Brockett:1970aa}, a sufficient condition for the existence of such a control is that the matrix
\begin{equation}\label{eq:gramian}
	W:=\int_0^1 BB^T du, \quad \text{where } B:=\beta [\nu_1 \ldots \nu_{n-1} ]
\end{equation}
should be non-singular (here $B$ is an $n\times (n-1)$ matrix with columns $\nu_i$, and $B^T$ its transpose), in which case a particular control which drives the solution from $x(0)$ to $x(1)$ is $\xi=B^T W^{-1}(x(1)-x(0))$.

Suppose $W$ is singular. Then there is a constant non-zero vector $a\in \R^n$ such that 
\[ 0=a^T W a=\int_0^1 (a^T B)(a^T B)^T\,du\]
which implies that $a^T B=\beta[\ip{a ,\nu_1},\ldots, \ip{a, \nu_{n-1}}]=0$ on $(0,1)$, i.e. $a$ is in the direction of $\pm T$. But this is impossible because $\gamma$ is closed and not constant, and $T$ is continuous.

We therefore have that $d\Phi_\gamma$ is surjective, and its kernel splits because it is a closed subspace of a Hilbert space. 
Therefore $\mathbf 0$ is a regular value of $\Phi$ and $\Omega$ is a submanifold (see e.g. \cite[Theorem~ 2.2.2]{Schrader:2016aa}, with \cite[Proposition~ 2.3]{Lang:1999sf}). 
To see that it is an \emph{analytic} submanifold we note that (e.g. in the proof of \cite[Proposition~ 2.3]{Lang:1999sf}) the charts for $\Omega$ are constructed by applying the inverse function theorem to $\Phi$. 
Since $\Phi$ is analytic as a consequence of Lemma~\ref{speed}, local inverses of $\Phi$ will also be analytic by a theorem of Whittlesey (\cite[p. 1081]{Whittlesey:1965aa}). 
So the charts are analytic, i.e. $\Omega$ is analytic.
\end{proof}

Combining Proposition \ref{energy-analytic-1} with Proposition \ref{proposition:Omega-210630}, we have: 
\begin{corollary}\label{restrictionanalytic}
The restriction $\E|\Omega$ is analytic.
\end{corollary}
Here we give a characterization of the tangent space $T_\gamma \Omega$. 
\begin{corollary} \label{Cor:210702-1}
The tangent space $T_\gamma\Omega$ is equal to $\ker d \Phi_\gamma$ and it consists of all $V\in H^2(S^1,\R^n)$ satisfying
\begin{equation}\label{omegatangent}
\ip{V_s ,T}+\frac{1}{\length(\gamma)}\int^{\length(\gamma)}_{0}\ip{V,\kappa} ds=0. 
\end{equation}
\end{corollary}
\begin{proof}
The relation $T_\gamma \Omega=\ker d\Phi_\gamma$ is a consequence of the regular values theorem (see e.g. \cite[Theorem~2.2.2]{Schrader:2016aa}). From \eqref{dphi} we see that 
\[
\ip{V_u,T} - \int_{S^1}\ip{V_u,T}du=0 
\]
for all $V \in T_\gamma \Omega$. Since $|\gamma_u|=\length(\gamma)$ for $\gamma \in \Omega$, multiplication by $1/\length(\gamma)$ and integration by parts gives the desired expression.
\end{proof}

\subsection{The gradient inequality} \label{subsection:LSineq}
\begin{proposition}\label{fredholm}
Let $\gamma$ be a stationary point of $\E|\Omega$, then $d^2(\E|\Omega)_\gamma$ is a Fredholm operator with index zero.
\end{proposition}
\begin{proof}
Let $\gamma \in \Omega$, and fix $V, W \in T_\gamma\Omega$ arbitrarily. 
Taking a derivative of \eqref{omegatangent}, we have 
\begin{equation}
\label{eq:210702-10}
\ip{V_{ss},T} = - \ip{V_s, \kappa }. 
\end{equation}
Combining Lemma \ref{L-second-var} with \eqref{eq:210702-10} and Corollary \ref{Cor:210702-1}, we reduce the second variation formula to 
\begin{align*}
d^2\E_\gamma (V,W) 
&= \int^{\length(\gamma)}_0 \begin{multlined}[t] \bigl[ \ip{W_{ss},2V_{ss}-6\ip{T,V_s}\kappa}\\
 +\ip{W_s,-6\ip{\kappa,V_{ss}}T+(15k^2-\lambda^2)\ip{V_s,T}T}\\
-(3k^2-\lambda^2)\ip{W_s,V_s} \bigr]\, ds. 
\end{multlined}
\end{align*}
Assuming furthermore that $\gamma $ is a critical point of $\E$ and therefore admits higher derivatives, we can differentiate 
\[
\partial_s(\ip{\kappa,V_s}T)=\ip{\gamma_{s^3},V_s}T+\ip{\kappa,V_{ss}}T+\ip{\kappa,V_s}\kappa
\]
and use this to eliminate the $\ip{\kappa,V_{ss}}T$ term from the expression for $d^2\E$. Then after further simplifications we find
\begin{align*}
d^2\E_\gamma (V,W) 
=\begin{multlined}[t] \int^{\length(\gamma)}_0 \bigl[ 2\ip{W_{ss},V_{ss}}  + \ip{W_s,6\ip{\gamma_{s^3},V_{s}}T+6\ip{T,V_s}\gamma_{s^3}}\\ 
+\ip{W_s,(15k^2-\lambda^2)\ip{V_s,T}T-(3k^2-\lambda^2)V_s} \bigr] \, ds. 
\end{multlined}
\end{align*}
Since $\gamma\in \Omega$ we have
\[ \ip{V,W}_{H^2}=\length^3(\gamma)\ip{V_{ss},W_{ss}}_{L^2(ds)}+\length(\gamma)\ip{V_s,W_s}_{L^2(ds)}+\ip{V,W}_{L^2}\]
 and observe that the second derivative of $\E$ has the form
\[
d^2\E_\gamma( V,W)=\frac{2}{\length(\gamma)^3}\ip{V,W}_{H^2} +\int^{\length(\gamma)}_0 \ip{\tau(V),W}\, ds, 
\]
where $\tau$ is a continuous linear map from $T_\gamma \Omega$ into $L^2(S^1,\R^n)$. 
We define the associated operator $B:T_\gamma\Omega\subset H^2(S^1,\R^n)\to T_\gamma\Omega^*$ by  
\[
B(V)=d^2(\E|\Omega)_\gamma(V,\cdot ). 
\]
Since $T_\gamma\Omega$ is a closed subspace in $H^2(S^1,\R^n)$, we see that $(T_\gamma\Omega, \langle \cdot, \cdot \rangle_{H^2})$ is a Hilbert space. 
Moreover, by the form of $d^2(\E|\Omega)_\gamma$ we observe that 
\[
B(V)=\dfrac{2}{\mathcal{L}(\gamma)^3}I(V) + T(V), 
\]
where $I$ is the Riesz map $I(V)=\ip{V,\cdot }_{H^2}$ and $T(V):=\int^{\mathcal{L(\gamma)}}_0 \ip{V,\tau(\cdot) }\,ds$. 
Employing the Riesz representation theorem in $(T_\gamma\Omega, \langle \cdot, \cdot \rangle_{H^2})$, 
we see that the operator $I$ is an isomorphism, and then it is Fredholm with index zero. 
As for $T$, recalling that $\tau : T_\gamma \Omega \to L^2(S^1, \R^n)$ is bounded, we have  
\[
\| T(V) \|_{(T_\gamma \Omega)^*} 
= \sup_{W \in T_\gamma \Omega, \, \| W \|_{H^2}=1} \Bigl| \int^{\mathcal{L(\gamma)}}_0 \ip{V,\tau(W) }\,ds \Bigr| 
\le C \| V \|_{L^2}
\]
for all $V \in T_\gamma \Omega$. 
Now it follows from the compactness of the imbedding $H^2 \subset L^2$ that a bounded sequence $(V_i) \subset T_\gamma \Omega$ has a subsequence converging in $L^2(S^1, \R^n)$, and then $(T V_i)$ has a convergent subsequence 
because the space $((T_\gamma \Omega)^*, \|\cdot\|_{(T_\gamma \Omega)^*})$ is a Banach space. 
Thus $T : T_\gamma \Omega \to (T_\gamma \Omega)^*$ is compact.
Then since $B$ is the sum of a Fredholm index zero operator and a compact operator, it is also Fredholm with index zero (e.g. \cite[Example~ 8.16]{Zeidler:1986aa}).	
\end{proof}

\begin{proposition}{\rm (\L ojasiewicz--Simon gradient inequality on $\Omega$)}.\label{LS1}
Let $ \varsigma \in\Omega $ be a stationary point of $\E$. There are constants $Z\in (0,\infty), \delta\in (0,1]$ and $\theta\in \left [\tfrac{1}{2},1\right )$ such that 
if $\alpha \in \Omega$ with $\norm{\alpha-\varsigma}_{H^2}<\delta$ then 
\[ 
\| d(\E|\Omega)_\alpha\|_{T_\alpha \Omega^*} \geq Z|\E(\alpha)-\E(\varsigma)|^\theta.
\]
\end{proposition}
\begin{proof}
Let $\phi:U\subset \Omega \to \phi(U)\subset B$ be a local chart for $\Omega$ with $\varsigma\in U$, 
where $B$ is a subspace of $H^2(S^1,\R^n)$, and choose $\delta$ such that $\alpha\in \Omega$ is also in $U$ when $\|\alpha-\varsigma\|_{H^2}<\delta$. 
Define $E:=\E\circ \phi^{-1}:\phi(U)\to \R$, then $dE=d\E\circ d\phi^{-1}$, and since $\varsigma$ is stationary we have 
\[
d^2 E_{\phi(\varsigma)} = d^2 \E_{\varsigma}(d\phi^{-1}_{\phi(\varsigma)}\cdot, d\phi^{-1}_{\phi(\varsigma)}\cdot).
\] 
Moreover, $\phi$ is analytic and $d\phi_\varsigma:T_\varsigma\Omega\to B$ is an isomorphism so it follows from Corollary~\ref{restrictionanalytic} and Proposition~\ref{fredholm} that $E$ is analytic and $d^2 E_{\phi(\varsigma)}$ is Fredholm. 
Therefore by \cite[Theorem 1]{Feehan:2020aa} there exist constants $\tilde Z\in (0,\infty),\tilde{\delta}\in (0,1], \theta\in \left [\tfrac{1}{2},1\right )$ such that 
if $\|\phi(\alpha)-\phi(\varsigma)\|_B<\tilde \delta$ then
\[ 
\|dE_{\phi(\alpha)}\|_{B^*} \geq \tilde Z \left | E(\phi(\alpha))-E(\phi(\varsigma))\right |^\theta =\tilde Z \left | \E(\alpha)-\E(\varsigma)\right |^\theta.  
\]
Since $d\phi^{-1}$ is continuous we have that there is a constant $c$ such that for any $\alpha\in \Omega$ with $\norm{\alpha-\varsigma}_{H^2}<\delta$ and any $V\in T_\alpha \Omega$ 
\[ 
\|V\|_{H^2} \leq c \|d\phi_\alpha (V)\|_B 
\]
and therefore 
\[
\|dE_{\phi(\alpha)}\|_{B^*}=\sup_{d\phi_\alpha (V)\in B}\frac{|dE_{\phi(\alpha)}d\phi_\alpha (V)|}{\|d\phi_\alpha (V)\|_B} 
\leq \sup_{V\in T_\alpha\Omega}\frac{|d\E_\alpha V|}{\frac{1}{c}\|V\|_{H^2}}
= c\|d(\E|\Omega)_\alpha\|_{{T_\alpha \Omega}^*}. 
\]
The existence of a $\delta$ such that if $\|\alpha-\varsigma\|_{H^2}<\delta$ then $\|\phi(\alpha)-\phi(\varsigma)\|_B<\tilde \delta$ follows from the continuity of $\phi$ and the fact that $\Omega$ is a submanifold of $H^2(S^1,\R^n)$.  
\end{proof}

\begin{lemma}\label{Pconts} 
The projection $P:\I^2(S^1,\R^n) \to \Omega$ which takes each $\gamma$ to its arc length proportional reparametrisation is continuous with respect to $H^2$ at any $\sigma \in \I^2(S^1,\R^n)$ which is stationary for $\E$.  
\end{lemma}
\begin{proof}
Write $w_\gamma(u)=\frac{1}{\length(\gamma)}\int_0^u |\gamma'(\tau)|\,d\tau$ so $\alpha(w):=P(\gamma)(w)=\gamma \circ w_\gamma^{-1}(w)=\gamma(u)$ and 
\begin{equation*}
\begin{aligned}
\alpha'(w)&=\length(\gamma)\frac{\gamma'(u)}{|\gamma'(u)|}=:\length(\gamma)T_\gamma (u), \\
\alpha''(w)&= \gamma''(u)\frac{\length(\gamma)^2}{|\gamma'(u)|^2}-\gamma'(u)\frac{\length(\gamma)^2}{|\gamma'(u)|^4}\ip{\gamma''(u),\gamma'(u)} =\length(\gamma)^2\kappa_\gamma(u).
\end{aligned}
\end{equation*}
Let $\sigma\in \I^2(S^1,\R^n)$ be a stationary point of $\E$ and $\gamma \in \I^2(S^1,\R^n)$ such that $\norm{\gamma-\sigma}_{H^2}<b/C_S$, where $b$ and $C_s$ are as in Lemma \ref{superlemma}. Then we have constants $c_1,c_2,c_3$ depending only on $\sigma$ such that Lemma \ref{superlemma}(i)-(iii) hold. 
Using
\begin{equation} \label{eq:derivative-w-invers}
\dfrac{d w^{-1}_\gamma}{dw}=\dfrac{\mathcal{{L}(\gamma)}}{|\gamma'(w)|}, 
\end{equation}
we obtain the following estimate on parameters
\begin{equation}
\label{eq:param}
\begin{aligned}
|u-\omega_\gamma^{-1}\circ \omega_\sigma(u)|
&=|\omega_\gamma^{-1}\circ \omega_\gamma(u)-\omega_\gamma^{-1}\circ \omega_\sigma(u)| \\
&=\Bigl| \int_{\omega_\sigma(u)}^{\omega_\gamma(u)} (\omega_\gamma^{-1})'(\tau)\, d\tau \Bigr|\\ 
&\leq \length(\gamma) \| |\gamma'|^{-1} \|_{L^\infty}  \int_0^u \left | \frac{|\gamma'(\tau)|}{\length(\gamma)} - \frac{|\sigma'(\tau)|}{\length(\sigma)} \right| d\tau  \\
&\leq \frac{1}{\length(\sigma)}\| |\gamma'|^{-1} \|_{L^\infty}\left ( \norm{\gamma'}_{L^\infty}+\length(\gamma) \right )\| \gamma'-\sigma' \|_{L^1}\\
&\leq \frac{2c_2}{c_1^2}\| \gamma'-\sigma' \|_{L^1}
\end{aligned}
\end{equation}
where we have also used $|\length(\gamma)-\length(\sigma)|\leq \norm{\gamma-\sigma}_{L^1}$.
Then 
\begin{align*}
\| P(\sigma)-P(\gamma) \|^2_{L^2}
&= \int_0^1 | \sigma\circ \omega_{\sigma}^{-1}(w)-\gamma\circ\omega_{\gamma}^{-1}(w)|^2\, dw \\
&\leq \begin{multlined}[t] 2 \int_0^1 |\sigma(u)-\sigma\circ \omega_{\gamma}^{-1}\circ \omega_{\sigma}(u)|^2 \frac{|\sigma'(u)|}{\length(\sigma)}\, du\\ + 2 \int_0^1 |\sigma\circ \omega_{\gamma}^{-1}(w) -\gamma\circ \omega_{\gamma}^{-1}(w)|^2\, dw	
 \end{multlined}
\\
&\leq \frac{2 \|\sigma'\|_{L^\infty}^3}{\length(\sigma)}\int_0^1 |u-\omega_{\gamma}^{-1}\circ \omega_{\sigma}(u)|^2\,du  +\frac{2 \|\gamma'\|_{L^\infty}}{\length(\gamma)}\| \sigma-\gamma \|_{L^2}^2\\
&\leq 8 c^5\| \gamma'-\sigma' \|_{L^1}^2
+ 2c\|\sigma-\gamma\|_{L^2}^2, 
\end{align*} 
where $c:=c_2/c_1$.
For the difference of first derivatives we have 
\begin{align*}
\| P(\sigma)'-P(\gamma)'\|_{L^2}^2
&= \int_0^1 |\length(\sigma)T_\sigma \circ \omega_\sigma^{-1}(w)- \length(\gamma) T_\gamma \circ \omega_\gamma^{-1} (w)|^2 \, dw \\
&\leq 2 \length(\sigma)^2\int_0^1 |T_\sigma\circ \omega_\sigma^{-1}(w)-T_\sigma\circ \omega_\gamma^{-1}(w)|^2\, dw \\ 
& \qquad + 2 \int_0^1 |\length(\sigma)T_\sigma\circ \omega_\gamma^{-1}(w)-\length(\gamma)T_\gamma \circ \omega_\gamma^{-1}(w) |^2\, dw. 
\end{align*}
For the first term on the right hand side of the inequality, using a change of variable, the fundamental theorem of calculus and \eqref{eq:param} we have
\begin{align*}
\length(\sigma)^2\int_0^1 &|T_\sigma\circ \omega_\sigma^{-1}(w)-T_\sigma\circ \omega_\gamma^{-1}(w)|^2\, dw \\
&\leq \length(\sigma)\|\sigma'\|_{L^\infty} \|T_\sigma'\|^2_{L^2} \int_0^1 |u-\omega_\gamma^{-1}\circ \omega_\sigma(u)|^2 \, du \\
&\leq 4c^4 c^2_2 c_3^2 \|\gamma'-\sigma'\|_{L^2}^2, 
\end{align*}
where we have used $T_\sigma'=\abs{\sigma'}\kappa_\sigma$.
For the second term, recalling \eqref{eq:Tlip}
\begin{align*}
& \int_0^1 |\length(\sigma)T_\sigma\circ \omega_\gamma^{-1}(w)-\length(\gamma)T_\gamma \circ \omega_\gamma^{-1}(w) |^2\, dw	\\
& \qquad \leq \frac{2 \|\gamma'\|_{L^\infty}}{\length(\gamma)} \int_0^1 \bigl[ |\length(\sigma)-\length(\gamma)|^2 + \length(\gamma)^2 |T_\sigma(u)-T_\gamma(u)|^2 \bigr] \, du\\
& \qquad \leq 2(c+2c^3) \|\sigma'-\gamma'\|_{L^2}^2. 
\end{align*}
Similarly for the difference of second derivatives
\begin{align*}
\|P(\sigma)''-P(\gamma)''\|_{L^2}^2
&=\int_0^1 |\length(\sigma)^2\kappa_\sigma\circ \omega_\sigma^{-1}(w)-\length(\gamma)^2 \kappa_\gamma\circ \omega_\gamma^{-1}(w)|^2 \, dw \\
&\leq \begin{multlined}[t] 2 \length(\sigma)^3 \|\sigma'\|_{L^\infty} \|\kappa_\sigma'\|_{L^2}^2 \int_0^1 |u-\omega_\gamma^{-1}\circ \omega_\sigma(u)|^2 \, du \\
 +4\frac{\|\gamma'\|_{L^\infty}}{\length(\gamma)} \|\kappa_\sigma\|_{L^2}^2 (\length(\sigma)+\length(\gamma))^2 \|\sigma'-\gamma'\|_{L^2}^2 \\
         +4\|\gamma'\|_{L^\infty} \length(\gamma)^3 \|\kappa_\sigma-\kappa_\gamma\|^2_{L^2} 
\end{multlined}\\
&\leq \begin{multlined}[t] 4cc_2^2 \left (\norm{\kappa_\sigma '}_{L^2}^2+2c_3^2 \right )\norm{\sigma'-\gamma'}_{L^2}^2 
+4c_2^4\norm{\kappa_\sigma-\kappa_\gamma}_{L^2}^2.
\end{multlined}
\end{align*}
We have assumed $\sigma$ is stationary so that, by the proof of Lemma \ref{critsmooth}, $\ltwo{\kappa_\sigma'}$ is bounded. 
Finally, we recall that by Lemma \ref{superlemma} (iii) $\kappa$ is Lipschitz, and this completes the proof.
\end{proof}


\begin{theorem}{\rm (Gradient inequality on $\I^2(S^1,\R^n)$)}\label{LS2}
Let $\sigma \in \I^2(S^1,\R^n)$ be a stationary point of $\E$. 
Then there are constants $Z\in (0,\infty), \delta \in (0,1]$ and $\theta\in \left [\tfrac{1}{2},1\right )$ such that if $\gamma \in \I^2(S^1,\R^n)$ with $\|\gamma-\sigma\|_{H^2}<\delta$ then 
\[ 
\|\grad \E_\gamma\|_{H^2(ds), \gamma} \geq Z |\E(\gamma)-\E(\sigma)|^\theta.
\] 
\end{theorem}
\begin{proof}
Let $\alpha,\varsigma\in \Omega$ be the respective arc length proportional reparametrisations of $\gamma,\sigma$. 
Then since $\E$ and $d\E$ are parametrisation invariant and $\|d(\E|\Omega)_\alpha\|_{T_\gamma\Omega^*}\leq \|d\E_\alpha\|_{{H^2}^*}$ we have by Proposition \ref{LS1}
\[ 
\| d\E_\alpha \|_{{H^2}^*} \geq \| d(\E|\Omega)_\alpha \|_{T_\gamma\Omega^*} \geq Z|\E(\alpha)-\E(\varsigma)|^\theta =Z|\E(\gamma)-\E(\sigma)|^\theta 
\]
provided $\|\alpha-\varsigma\|_{H^2}$ is sufficiently small,  which can be arranged according to Lemma~\ref{Pconts} because $\sigma$ is stationary.  
Since reparametrisation is a linear map on $H^2$ we have  $d\E_\alpha(V) =d\E_\gamma (V \circ \omega_\gamma)$ and then
\begin{equation}\label{supnormalpha}
\begin{aligned}
\|d\E_\alpha\|_{{H^2}^*}
&=\sup_{\|V\|_{H^2}=1} |d\E_\gamma (V\circ \omega_\gamma)| \\
&= \sup_{\|V\|_{H^2}=1} \ip{\grad \E_\gamma,V\circ \omega_\gamma}_{H^2(ds),\gamma}. 
\end{aligned}
\end{equation}
From \eqref{eq:derivative-w-invers} 
we calculate
\begin{equation}\label{Vcirc}
\| V\circ \omega_\gamma \|^2_{H^2(ds),\gamma} = \length(\gamma) \|V\|_{L^2}^2 + \dfrac{1}{\length(\gamma)}\|V'\|_{L^2}^2 + \frac{1}{\length(\gamma)^3} \|V''\|_{L^2}^2.  
\end{equation} 
By Lemma \ref{superlemma} we have upper and lower bounds for $\length(\gamma)$ and therefore a constant $c_1$ such that 
\[ 
\|d\E_\alpha\|_{{H^2}^*} \leq c_1 \|\grad \E_\gamma \|_{H^2(ds)}. 
\]
\end{proof}

\section{Convergence} \label{subsection:full-limit-convergence}
Since the \L ojasiewicz--Simon gradient inequality proved above only holds in an $H^2$-neighbourhood of a critical point we will need subconvergence in $H^2$ of minimizing sequences in order to use it. 
To this end we will prove a Palais--Smale type condition for $\E|\Omega$ by adapting the method used in \cite{Schrader:2016ab}.  

First let us define an auxiliary functional $J:\I^2(S^1,\R^n)\to \R$ by
\begin{equation}\label{J}
	J(\gamma):=\frac{1}{\length(\gamma)^3}\int_0^1\abs{\gamma''}^2\, du+\lambda^2 \length(\gamma)\, .
\end{equation}
This function has the property $J|\Omega=\E|\Omega$, and also the following.

\begin{lemma}\label{loccoer}
$J$ is locally coercive modulo $C^1$ in the following sense. 
Let $U$ be an open neighbourhood in $\I^2(S^1,\R^n)$ for which there are constants $c_1,c_2$  such that for all $\gamma \in U$: $\|\gamma\|_{H^2}<c_1$ and $0<c_2<\abs{\gamma'(u)}$. Then there exist positive constants $c_3$ and $c_4$ depending on $U$ such that for any $V\in H^2(S^1,\R^n)$ 
\begin{align}\label{coercive}
d^2J_\gamma(V,V) \geq c_3 \|V\|_{H^2}^2 - c_4 \|V\|_{C^1}^2. 
\end{align}
\end{lemma}
\begin{proof}
To calculate the second derivative of $J$ first calculate $$d\length_\gamma V=\int_0^1\ip{V',T}\, du$$ and then
\[ d^2\length_\gamma (V,V)=\int_0^1\frac{1}{\abs{\gamma'}}\abs{V'}^2-\frac{1}{\abs{\gamma'}^3}\ip{V',\gamma'}^2 \, du. 
\]
Using the assumptions on $U$ there are constants $\bar c, \tilde c$ such that 
\begin{equation}\label{eq:dlength}
\begin{aligned}
\abs{d\length_\gamma V}&\leq \bar c\norm{V'}_{L^1}, \\
\abs{d^2\length_\gamma(V,V)}&\leq \tilde c\norm{V'}_{L^2}^2. 
\end{aligned}
\end{equation}
The derivatives of $J$ are
\begin{align}\label{eq:dj}
	dJ_\gamma V&=\frac{-3}{\length(\gamma)^4}d\length_\gamma(V)\int_0^1\abs{\gamma''}^2\, du+\frac{2}{\length(\gamma)^3}\int_0^1\ip{V'',\gamma''}\, du+\lambda^2 d\length_\gamma V,
\end{align}
and
\begin{align*}
d^2J_\gamma &(V,V)\\
&=-\frac{12}{\length(\gamma)^5}(d\length_\gamma V)^2\norm{\gamma''}_{L^2}^2-\frac{3}{\length(\gamma)^4}d^2\length_\gamma(V,V)\norm{\gamma''}_{L^2}^2\\ \nonumber
&\qquad -\frac{12}{\length(\gamma)^4}d\length_\gamma V\int_0^1\ip{V'',\gamma''}\, du+\frac{2}{\length(\gamma)^3}\norm{V''}_{L^2}^2+\lambda^2d^2\length_\gamma(V,V), 
\end{align*}
and then using \eqref{eq:dlength} and the assumptions on $U$ again, there are positive constants $a_1,a_2, a_3$ such that
\begin{align*}
d^2J_\gamma(V,V) &\geq a_1\norm{V''}_{L^2}^2-a_2\norm{V'}_{L^2}^2-a_3\norm{V'}_{L^2}\norm{V''}_{L^2}\\
&\geq a_1\norm{V}_{H^2}^2-(a_1+a_2)\norm{V}^2_{H^1}-a_3\norm{V'}_{L^2}\norm{V''}_{L^2}. 
\end{align*}
If we apply the inequality $2ab\leq \varepsilon a^2+\frac{1}{\varepsilon}b^2, \varepsilon>0$ to the last term, choosing $\varepsilon$ sufficiently large and using a Sobolev imbedding we obtain \eqref{coercive}.
\end{proof}

\begin{corollary}\label{corcoer} If $U\subset  \I^2(S^1,\R^n)$ satisfies the same conditions  as in Lemma \ref{loccoer} and is convex, then there exist constants $c_1,c_2$ such that for any $\gamma,\beta \in U$ 
	\begin{equation}\label{coercive2}
\bigl( dJ_\beta - dJ_\gamma \bigr)(\beta-\gamma) \geq c_1 \|\beta-\gamma\|_{H^2}^2 - c_2 \|\beta-\gamma\|_{C^1}^2. 
\end{equation}
\end{corollary}
\begin{proof}
	Since $\gamma+t(\beta-\gamma)\in U $ for all $t\in[0,1]$, \eqref{coercive} holds and
\begin{align*}
\bigl( dJ_{\beta}-dJ_\gamma  \bigr)(\beta-\gamma) &=\int_0^1 \frac{d}{dt}dJ_{\gamma+t(\beta-\gamma)}(\beta-\gamma)\, dt \\
&=\int_0^1 d^2J_{\gamma+t(\beta-\gamma)}(\beta-\gamma,\beta-\gamma)\,dt \\
&  \geq c_1 \|\beta-\gamma\|_{H^2}^2 - c_2 \|\beta-\gamma\|_{C^1}. 
\end{align*} 
\end{proof}

Next we construct a continuous projection onto $T_\alpha\Omega$ using a right inverse for the map $d\Phi$ from Section \ref{subsection:al-para-curves}. 
For this we just need to choose initial conditions for the construction described in the proof of Lemma \ref{proposition:Omega-210630}. 
We define $r_\alpha: H^2_{zm}(S^1,\R)\to H^2(S^1,\R^n)$ by
\begin{align}\label{eq:r}
r_\alpha w &:=y-x,\quad \text{where}\\ \nonumber
y'&= w\frac{\alpha'}{\abs{\alpha'}},\quad y(0)=0, \\ \nonumber 
x'&= B \xi ,\quad x(0)=0, \quad \xi=B^T W^{-1}y(1),
\end{align}
and $B$ and $W$ are the matrices from \eqref{eq:gramian}. Let us confirm that $r$ has the desired properties. 
It follows from \eqref{dphi} that 
\begin{equation}\label{eq:rightinverse}
	d\Phi_\alpha r_\alpha w=\ip{\left(r_\alpha w\right )', \frac{\alpha'}{\abs{\alpha'}}}-\int_0^1 (r_\alpha w)' \frac{\alpha'}{\abs{\alpha'}}\, du=w-\int_0^1 w\, du=w. 	
\end{equation}
Moreover, $r_\alpha w(0)=0$ and 
\begin{align*}
 r_\alpha w(1)=y(1)-\int_0^1 x' \, du& =y(1)-\int_0^1 BB^TW^{-1} y(1)\,du \\ &=y(1)-WW^{-1} y(1)\\ &=0. 
\end{align*}
Thus $r_\alpha w$ is periodic. Periodicity of $(r_\alpha w)'$ follows from that of $\alpha', w$ and $\beta$. 

Now for each $\alpha\in \Omega$ we define $\pr_{T_\alpha\Omega}:H^2(S^1,\R^n)\to T_\alpha\Omega$ by 
\begin{equation}\label{eq:projection}
	\pr_{T_\alpha\Omega}V:=(1-r_\alpha d\Phi_\alpha )V.
\end{equation}
Indeed $d\Phi_\alpha \pr_{T_\alpha\Omega} V=d\Phi_\alpha V -d\Phi_\alpha r_\alpha d\Phi_\alpha V=0$ by \eqref{eq:rightinverse} and so $\pr_{T_\alpha\Omega}\in \ker d\Phi_\alpha=T_\alpha \Omega$.


The following lemma will be used in Proposition \ref{PSC} to estimate terms involving the projection onto $T_\alpha\Omega$. Its proof requires some estimates on the matrix $W$, which we have included in Appendix \ref{appendix}. 

\begin{lemma} \label{estdjpr}
Let $U$ be an $H^2$-bounded subset of $\Omega$ for which there exists $c_1 $  such that   $0<c_1<\abs{\alpha'(u)}$ for all $\alpha \in U$. Then there exists a constant $c$ such that for all $\alpha\in U$ and $v\in H^2$
\begin{equation}
\abs{dJ_\alpha r_\alpha d\Phi_\alpha v }\leq c\norm{v}_{C^1}. 
\end{equation}
\end{lemma}
\begin{proof}
For this first we note from \eqref{eq:r} that $\abs{y(1)}\leq \norm{w}_{L^\infty}$, and since the $\nu_i$ are orthonormal, we deduce from Lemma \ref{lemma:W^-1-bdd-2} that
\[  \abs{(r_\alpha w)'}\leq \abs{w}+\norm{W^{-1}}\abs{y(1)}\leq c \norm{w}_{L^\infty}. 
\]
Moreover, it follows from \eqref{dphi} that $\abs{d\Phi_\alpha v}\leq 2\norm{v'}_{L^\infty}$, and then 
\begin{equation}\label{eq:rphidash}
\abs{(r_\alpha d\Phi_\alpha v)'}\leq c\norm{d\Phi_\alpha v}_{L^\infty}\leq 2c\norm{v'}_{L^\infty}.
\end{equation}
From \eqref{eq:r} again, recalling that $\abs{\alpha'}=\length(\alpha)$, calculate the second derivative 
\begin{equation}\label{eq:ddr}
	(r_\alpha d\Phi_\alpha v)''=(d\Phi_\alpha v)'T+d\Phi_\alpha v \frac{\alpha''}{\length(\alpha)}+(B\xi)' \,.
\end{equation} 
Note that because $\abs{\alpha'}$ is constant $\ip{T,\alpha''}=0$ and then from \eqref{eq:dj},
\begin{align}\label{djpr} \nonumber
		dJ_\alpha & r_\alpha d\Phi_\alpha v \\ \nonumber
&=\frac{-3}{\length(\alpha)^4}d\length_\alpha(r_\alpha d\Phi_\alpha v)\int_0^1\abs{\alpha''}^2\, du+\frac{2}{\length(\alpha)^3}\int_0^1\ip{(r_\alpha d\Phi_\alpha v)'',\alpha''}\, du\\ \nonumber 
&\qquad \qquad+\lambda^2 d\length_\alpha r_\alpha d\Phi_\alpha v \\ 
&=\frac{-3}{\length(\alpha)^4}d\length_\alpha(r_\alpha d\Phi_\alpha v)\norm{\alpha''}^2_{L^2}
+\frac{2}{\length(\alpha)^3}\int_0^1\frac{d\Phi_\alpha v}{\length(\alpha)}\abs{\alpha''}^2+\ip{(B\xi)',\alpha''}\, du\\ \nonumber
&\qquad \qquad+\lambda^2 d\length_\alpha r_\alpha d\Phi_\alpha. 
\end{align}
From the definition \eqref{eq:gramian} of $B$ we have \[ (B\xi)'=\beta'\sum_i\nu_i\xi_i+\beta\sum_i\nu_i'\xi_i+B\xi', \]
and from \eqref{frame}, $\ip{\nu_i',\alpha''}=0$, so recalling that $\beta$ is smooth and the $\nu_i$ are normalised, we have
\[
\abs{\ip{(B\xi)',\alpha''}}\leq c_1\abs{\xi}+c_2\abs{\xi'}. 
\]
Now as in \eqref{eq:rphidash}, we have $\abs{\xi}\leq \norm{W^{-1}}\abs{y(1)} \leq 2c \norm{v'}_{L^\infty}$. 
Since
\[\xi'=\left(\beta'[\nu_1\ldots \nu_{n-1}]^T+\beta [\nu_1' \ldots \nu_{n-1}']^T\right ) W^{-1} y(1),  
 \]
using \eqref{frame} to estimate $\nu_i'$, we obtain 
\[
\abs{\xi'}\leq \left (c_1+c_2\abs{\alpha''}\right )\norm{v'}_{L^\infty}. 
\]
Thus it follows that $\abs{\ip{(B\xi)',\alpha''}}\leq \left (c_1+c_2\abs{\alpha''}\right )\abs{\alpha''}\norm{v'}_{L^\infty}$. 
Combining this into \eqref{djpr} with the estimates \eqref{eq:dlength}, \eqref{eq:rphidash}, and $\abs{d\Phi_\alpha v}\leq \norm{v'}_{L^\infty}$, we see that
\begin{align*}
\abs{dJ_\alpha r_\alpha d\Phi_\alpha v }\leq c\norm{v'}_{L^\infty}. 
\end{align*}
Therefore Lemma \ref{estdjpr} follows. 
\end{proof}

Now we can prove the following Palais--Smale type condition for $\E|\Omega$ (it is not quite the standard condition because we need to assume an $L^2$-bound on the sequence).

\begin{proposition}\label{PSC}
Let $(\alpha_i)$ be a sequence of curves in $\Omega$ such that $\E(\alpha_i)$ and $\| \alpha_i \|_{L^2}$ are bounded, and $\norm{d(\E_{\alpha_i})}\to 0$. Then $(\alpha_i)$ has a subsequence converging in $H^2$. 
\end{proposition}
\begin{proof}
The assumed bounds on $\E(\alpha_i)$ and $\|\alpha_i\|_{L^2}$ together imply an $H^2$-bound and then by compactness of the Sobolev imbedding $H^2 \subset C^1$ we have a subsequence, still denoted $(\alpha_i)$, which converges in $C^1$ to some $\alpha_\infty$. 
Let $\varepsilon>0$ be sufficiently small so that the open $C^1$-ball $B_\varepsilon^{C^1}(\alpha_\infty)$ contains only immersions (c.f. Lemma~\ref{superlemma}) and choose $N$ sufficiently large that for all $i>N$, $\alpha_i\in B_\varepsilon^{C^1}(\alpha_\infty)$. Then for $i>N$, which we assume from now on, $\alpha_i$ is contained in a neighbourhood $U$ satisfying the assumptions of Corollary~\ref{corcoer}. Hence from \eqref{coercive2} we get 
\[ 
\bigl( dJ_{\alpha_j}-dJ_{\alpha_i} \bigr)(\alpha_j-\alpha_i) \geq c_1 \|\alpha_i-\alpha_j\|_{H^2}^2 - c_2 \|\alpha_i-\alpha_j\|_{C^1}^2\,
\]
with $c_1>0$. From \eqref{eq:projection} and $J|\Omega=\E|\Omega$
\begin{align*} 
\bigl( dJ_{\alpha_j}-dJ_{\alpha_i} \bigr)(\alpha_j-\alpha_i)&=\bigl( d\E_{\alpha_j}\pr_{T_{\alpha_j}\Omega}-d\E_{\alpha_i}\pr_{T_{\alpha_i}\Omega} \bigr)(\alpha_j-\alpha_i)\\
&\qquad +(dJ_{\alpha_j}r_{\alpha_j} d\Phi_{\alpha_j}-dJ_{\alpha_i}r_{\alpha_i} d\Phi_{\alpha_i})(\alpha_j-\alpha_i). 
\end{align*}
Then rearranging the inequality
\begin{align*}
	c_1 \|\alpha_i-\alpha_j\|_{H^2}^2& \leq c_2 \|\alpha_i-\alpha_j\|_{C^1}^2 +\bigl( d\E_{\alpha_j}\pr_{T_{\alpha_j}\Omega}-d\E_{\alpha_i}\pr_{T_{\alpha_i}\Omega} \bigr)(\alpha_j-\alpha_i)\\
&\qquad +(dJ_{\alpha_j}r_{\alpha_j} d\Phi_{\alpha_j}-dJ_{\alpha_i}r_{\alpha_i} d\Phi_{\alpha_i})(\alpha_j-\alpha_i). 
\end{align*}
Now by Lemma \ref{estdjpr}, since $\norm{d(\E_{\alpha_i})}\to 0$ by assumption and $\alpha_i$ converges in $C^1$, 
the right hand side of the above inequality converges to zero and $\alpha_i$ converges in $H^2$. 
\end{proof}

\begin{theorem} \label{theorem:5.5}
Let $\gamma$ be a solution to the $H^2(ds)$-gradient flow of the modified elastic energy $\E$. 
Then there is a stationary point $\gamma_\infty\in H^2(S^1,\R^n)$ such that $\gamma(t)\to \gamma_\infty$ in $H^2$ as $t \to \infty$.
\end{theorem}
\begin{proof}
Consider the projected and translated flow 

\[
\alpha(t):=P(\gamma(t)) - \frac{1}{\length(\gamma(t))}\int^{\length(\gamma(t))}_0 \gamma(t) \, ds,
\] 
where as in Lemma~\ref{Pconts}, $P(\gamma(t))$ is the arc length proportional reparametrisation of $\gamma(t)$. 
From parametrisation and translation invariance of the energy we have $\lambda^2 \length(\alpha)<\E(\alpha)=\E(\gamma)\leq \E(\gamma(0))$. 
Moreover, using the estimates in Lemma~\ref{Pconts} and the Poincar\'e--Wirtinger inequality, we see that $\|\alpha(t) \|_{L^2}$ is also bounded.
From \eqref{l2time} (with $T\to \infty$) there exists a monotone, divergent sequence $(t_i)$ such that \[\norm{\grad \E(\gamma(t_i))}_{H^2(ds)}\to 0 .\] 
Then \eqref{supnormalpha} and \eqref{Vcirc} imply that 
\[ 
\|d\E_\alpha\|_{{H^2}^*}\leq \Bigl( \length(\gamma) + \frac{1}{\length(\gamma)}+ \frac{1}{\length(\gamma)^3}\Bigr)^{1/2} \|\grad \E_\gamma \|_{H^2(ds)}. 
\]
Since $\length(\gamma)<\E(\gamma(0))/\lambda^2$, and applying the H\"older inequality to Fenchel's theorem $2\pi\leq \int |k| ds$ gives
\[ 
\frac{1}{\length(\gamma)}\leq \frac{1}{4\pi^2}\int k^2 ds < \frac{1}{4\pi^2}\E(\gamma(0)), 
\]
it follows that $\|d\E_{\alpha(t_i)}\|_{{H^2}^*}\to 0$ too. 
Hence $\alpha(t_i)$ which we abbreviate to $\alpha_i$ satisfies the assumptions of Proposition \ref{PSC} and there exists a subsequence, still denoted $\alpha_i$, converging in $H^2$ to a stationary point $\alpha_\infty$. 
Now by Theorem~\ref{LS2} there are constants $Z>0$, $\delta\in (0,1]$, and $\theta\in [1/2,1)$ such that for any $x\in \I^2$ with $\|x-\alpha_\infty\|_{H^2}<\delta$:
\begin{equation} \label{LSineq}
\|\grad \E_x\|_{H^2(ds)} \geq Z \abs{\E(x)-\E(\alpha_\infty)}^\theta. 
\end{equation}
Since the $H^2(ds)$-Riemannian distance and the standard $H^2$ metric are equivalent (Section \ref{metrics}), there exist $\tilde \delta>0, r>0$ such that $B^{H^2}_{\tilde \delta}(\alpha_\infty)\subset B^{\dist}_r(\alpha_\infty)\subset B^{H^2}_\delta(\alpha_\infty)$.
For any $i$ such that $\alpha_i\in B^{H^2}_{\tilde \delta}(\alpha_\infty) $ we let $\beta_i(t)$ be the $H^2(ds)$-gradient flow with intial data $\beta_i(t_i)=\alpha_i$. 
Then due to the uniqueness of the flow, for all $t>t_i$, $\beta_i(t)$ is a fixed (i.e. time independent) reparametrisation and translation of $\gamma(t)$ , namely $\beta_i(t)=\gamma(t)\circ \omega_{\gamma(t_i)}^{-1}-\frac{1}{\length(\gamma(t_i))}\int^{\length(\gamma(t_i))}_0 \gamma(t_i) \, ds$, 
and therefore by \eqref{gradinvariance} we have 
\begin{equation}
\label{eq:210702-1}
\|\grad \E_{\beta_i(t)}\|_{H^2(ds)}=\|\grad \E_{\gamma(t)}\|_{H^2(ds)}. 
\end{equation}
It follows that the trajectories $\beta_i(t)$ and $\gamma(t)$ have the same $H^2(ds)$-length. 
Let $T_i$ be the maximum time such that $\| \beta_i(t)-\alpha_\infty \|_{H^2}<\tilde \delta$ for all $t\in [t_i,T_i)$. 
Define
\[ 
H(t):=(\E(\gamma(t))-\E(\alpha_\infty))^{1-\theta}.  
\]
Then $H>0$ and is monotonically decreasing because $\E(\alpha)=\E(\gamma)$. 
Since \eqref{LSineq} holds for $\beta_i(t)$ with $t\in [t_i,T_i)$, 
we observe from $\E(\beta_i(t))=\E(\gamma(t))$ and~\eqref{eq:210702-1} that 
\begin{align*}
-H'(t) &= -(1-\theta)\left (\E(\gamma(t))-\E(\alpha_\infty)\right )^{-\theta} \frac{d\E(\gamma(t))}{dt} \\
&=(1-\theta)\left (\E(\gamma(t))-\E(\alpha_\infty)\right )^{-\theta} \| \grad \E_\gamma\|^2_{H^2(ds)}  \\
& \geq (1-\theta) Z \|\grad \E_\gamma\|_{H^2(ds)}. 
\end{align*}
Integrating over $[t_i,T_i)$ we get 
\[ 
(1-\theta) Z \int_{t_i}^{T_i} \| \grad \E_\gamma \|_{H^2(ds)}\,dt\leq H(t_i)-H(T_i).  
\]
Now if we fix a $j$ such that $\|\alpha_j-\alpha_\infty\|_{H^2}<\tilde \delta$ and let $W:=\cup_{i\geq j}[t_i,T_i)$ we have that 
\begin{equation}\label{eq:finite}
\int_W \| \grad \E_\gamma\|_{H^2(ds)} \,dt \leq \dfrac{H(t_i)}{(1-\theta)Z} . 
\end{equation}
In fact, there exists $N \in \mathbb{N}$ such that $\| \beta_N(t)-\alpha_\infty\|_{H^2}<\tilde\delta$ for \emph{all} $t>t_N$. 
If not, then for each $i \in \mathbb{N}$ there exists $ T_i$ such that $\beta_i( T_i)$ is on the boundary of the ball $B^{H^2}_{\tilde \delta}(\alpha_\infty)$, and there exists a subsequence, still denoted $(t_i)$, such that the intersection $\cap_{i\geq j}[t_i,T_i)$ is empty.
By the choice of $\tilde \delta$, Lemma~\ref{bruv1} applies and there is a $C>0$, depending only on $\alpha_\infty$ and $r$, such that
\begin{align}\nonumber
\tilde{\delta} = \|\beta_i( T_i)-\alpha_\infty\|_{H^2} 
&\leq  \|\beta_i(t_i)-\alpha_\infty\|_{H^2} + \|\beta_i(t_i)-\beta_i( T_i)\|_{H^2} \\ \label{deltabound}
&\leq \|\alpha(t_i)-\alpha_\infty\|_{H^2} + C \dist(\beta_i(t_i),\beta_i( T_i))\\ \nonumber
&\leq \|\alpha(t_i)-\alpha_\infty\|_{H^2} + C \int_{t_i}^{ T_i} \|\gamma_t\|_{H^2(ds)}\, dt, 
\end{align}
where we have used \eqref{eq:210702-1}. 
But then the integral $\int_W \| \grad \E_\gamma\|_{H^2(ds)}\,dt$ cannot be finite, contradicting \eqref{eq:finite}. 
So there exists $N \in \mathbb{N}$ such that $\beta_N(t)\in B_{\tilde \delta}^{H^2}(\alpha_\infty)$ for all $t>t_N$ and therefore 
\[
	\int_{t_N}^\infty \|\gamma_t\|_{H^2(ds)} \, dt <\infty,  
\]
that is, the $H^2(ds)$-length of $\gamma(t)$ is finite. Hence by Lemma~\ref{finitelengthlimit} the flow converges in the $H^2(ds)$-distance, and therefore also in $H^2$ (Lemma~\ref{bruv1}).
\end{proof}

We are now in a position to prove Theorem \ref{maintheorem}. 

\begin{proof}[Proof of Theorem {\rm \ref{maintheorem}}]
By Proposition \ref{shortexistence} we see that problem \eqref{eq:GF} possesses a unique local-in-time solution $\gamma \in C^1([0,T), \I(S^1, \R^n))$ for some $T>0$. 
Proposition \ref{git-existence} extends the local-in-time solution into a global-in-time solution. 
Finally, we observe from Theorem \ref{theorem:5.5} that the global-in-time solution converges to an elastica as $t \to \infty$ in the $H^2$-topology.  
\end{proof}

\begin{remark}

As in \cite{Langer-Singer_1985} the classification of closed elastica in $\R^2$ and $\R^3$ allows us to determine the shape of the limit in these cases. In $\R^2$ the only closed elastica are the (geometric) circle, the figure eight elastica and their multiple covers. Since the flow must remain in a path component of $\I^2(S^1,\R^2)$, for rotation index $p>0$ the limit is a $p$-fold circle, and an initial curve with rotation index $0$ will converge to a (possible multiply covered) figure eight. In $\R^3$ there are more closed elastica, however it was proved in \cite{Langer-Singer_1985} that circles are the only \emph{stable} closed elastica. 
In both $\R^2$ and $\R^3$ it follows from \eqref{critsmooth} that the limiting circle has radius $\abs{\lambda}^{-1}$.
\end{remark}

\begin{remark}
It was mentioned in the introduction that in \cite{DallAcqua:2016aa} and \cite{MP_2021} the convergence results for the $L^2(ds)$-gradient flow of elastic energy are modulo reparametrisation. They are obtained by proving a \L ojasiewicz--Simon gradient inequality for the $L^2$-norm of the gradient in a space of graphs over the critical point, which is different to the approach taken here. However, it seems that the key difference is not in the method, but in the fact that the $L^2(ds)$-metric does not dominate $L^2$ without extra assumptions on parametrisation, and therefore it would not be possible to verify an $L^2$ version of \eqref{deltabound}.

\end{remark}

\begin{appendix}

\section{Auxiliary lemmata} \label{appendix}
In this appendix, we prove some estimates on the matrix $W$ defined by \eqref{eq:gramian} that are used in the proof of Lemma \ref{estdjpr}. 

\begin{lemma} \label{lemma:W^-1-bdd-1}
Let $\gamma \in \I^2(S^1,\R^n)$. Then there exists a constant $C>0$ such that 
\[
\| W^{-1} \| < C, 
\]
where $W$ is defined by \eqref{eq:gramian}. 
\end{lemma}
\begin{proof}
Let $\{ T(u), \nu_1(u), \ldots, \nu_{n-1}(u) \}$ be an orthonormal basis at $\gamma(u)$ for $u \in S^1$ (e.g. as in \eqref{frame}).
Let $\mu \in \R$ be the smallest eigenvalue of $W$ and $\eta$ a corresponding eigenvector with  $|\eta|=1$.
Since 
\[
\eta^T W \eta = \mu \eta^T \eta = \mu, 
\]
we have 
\[
\mu = \int^1_0 \eta^T B B^T \eta\, du
= \int^1_0 | B^T \eta|^2 \, du
= \sum^{n-1}_{j=1} \int^1_0 \beta^2 \ip{\eta, \nu_j}^2 \, du,  
\]
where we used the definition of $B$ \eqref{eq:gramian}.  
Recalling that 
\[
1 = |\eta|^2 = \ip{\eta, T}^2 + \sum^{n-1}_{j=1} \ip{\eta, \nu_j}^2, 
\]
we see that 
\begin{equation}
\label{eq:1st-eigen}
\mu = \int^1_0 \beta^2 \bigl( 1 - \ip{\eta,T}^2 \bigr) \, du.  
\end{equation}
If $\mu=0$, then it follows from \eqref{eq:1st-eigen} that $\ip{\eta , T(u)} \equiv 1$ and $T(u)=\pm \eta$ for all $u$. But this is impossible because $\gamma$ is closed and $T$ is continuous. Hence
\begin{equation}\label{eq:Wbound}
\| W^{-1} \|:= \sup_{v \in \R^n} \frac{| W^{-1} v|}{\abs{v}} = \sup_{x \in \R^n}\frac{\abs{x}}{\abs{Wx}}= \dfrac{1}{\mu} < \infty. 
\end{equation}
Therefore Lemma~\ref{lemma:W^-1-bdd-1} follows. 
\end{proof}

\begin{lemma} \label{lemma:W^-1-bdd-2}
Let $U\subset \Omega$ such that $\norm{\alpha}_{H^2}<\infty$ for all $\alpha\in U$.
Then there exists a constant $C>0$ such that 
\[
\| W_\alpha^{-1} \| < C
\]
for all $\alpha\in U$, where $W_\alpha$ denotes the matrix defined by \eqref{eq:gramian} with $\gamma=\alpha$.  
\end{lemma}
\begin{proof} Suppose there is no such $C$, and therefore there exists a sequence $(\alpha_i)\subset U$ such that $\norm{W_{\alpha_i}^{-1}}\to \infty$. Equivalently, abbreviating to $W_i$ and letting $\mu_i$ be the smallest eigenvalue of $W_i$, by \eqref{eq:Wbound} $\mu_i\to 0$. We will show that this leads to a contradiction.

Since the sequence $(\alpha_i)$ is bounded in $H^2(S^1,\R^n)$, there exists $\alpha_\infty \in H^2(S^1,\R^n) \cap C^{1+\delta}(S^1,\R^n)$ such that 
\begin{equation}
\label{conv-1}
\begin{aligned}
& \alpha_i \rightharpoonup \alpha_\infty \quad \text{weakly in} \quad H^2(S^1,\R^n), \\
& \alpha_i \to \alpha_\infty \quad \text{in} \quad C^{1+\delta}(S^1,\R^n), 
\end{aligned}
\end{equation}
up to a subsequence, where $\delta \in (0, 1/2)$. 
This implies that 
\[
T_i := \dfrac{\alpha_i'}{|\alpha_i'|} \to T_\infty := \dfrac{\alpha_\infty'}{|\alpha_\infty'|} \quad \text{in} \quad C^\delta(S^1,\R^n). 
\]
As in \eqref{frame} we construct an orthonormal frame $\{ T_i, \nu_i^1, \ldots, \nu_i^{n-1} \}$ along $\alpha_i$ by starting with an orthonormal basis $\{T_i(0),\nu_{i,0}^1,\ldots,\nu_{i,0}^{n-1}\}$ and letting
 $\{ \nu_i^1, \ldots, \nu_i^{n-1} \}$ be the solutions of 
\begin{equation}
\label{eq:220314-2}
 (\nu_i^j)'= - \dfrac{1}{|\alpha_i' |^2} \langle (\nu_i^j), \alpha_i'' \rangle \alpha_i', \quad 
\nu_i^j(0) = \nu_{i,0}^j 
\end{equation} 
for $j \in \{1, \ldots, n-1\}$. 
 
Since $|\nu_i^j| \equiv 1$, choosing $N$ sufficiently large so that for $i>N$, $\alpha_i$ is contained in a $C^1$-ball centred at $\alpha_\infty$ and therefore $|\alpha_i'(u)|$ is bounded below, we have that 
\[
\| (\nu_i^j)' \|_{L^2} \le C \| \alpha_i'' \|_{L^2(S^1)} < C  
\]
for all $i >N$ and $j \in \{ 1, \ldots, n-1\}$. 
Thus $( \nu^j_i )$ is bounded in $H^1(S^1,\R^n)$, and then we find $\{ \nu_\infty^1, \ldots, \nu_\infty^{n-1}\}$ such that 
\begin{equation}
\label{conv-2}
\begin{aligned}
& \nu_i^j \rightharpoonup \nu_\infty^j \quad \text{weakly in} \quad H^1(S^1,\R^n), \\
& \nu_i^j \to \nu_\infty^j  \quad \text{in} \quad C^{\delta}(S^1),  
\end{aligned}
\end{equation}
up to a subsequence. We note here that $|\nu_{\infty}^j| \equiv 1$. 
By \eqref{eq:220314-2} we have 
\begin{equation}
\label{eq:220314-3}
\nu_i^j(u) = \nu_i^j(0) - \int^u_0 \dfrac{1}{| \alpha_i'(\tilde{u}) |^2} \langle \nu_i^j(\tilde{u}), \alpha_i''(\tilde{u}) \rangle \alpha_i'(\tilde{u}) \, d\tilde{u}. 
\end{equation}
Letting $i \to \infty$ in \eqref{eq:220314-3}, we observe from \eqref{conv-1} and \eqref{conv-2} that 
\[
\nu_\infty^j(u) = \nu_\infty^j(0) 
- \int^u_0 \dfrac{1}{|\alpha_\infty'(\tilde{u}) |^2} \langle \nu_{\infty}^j(\tilde{u}), \alpha_\infty''(\tilde{u}) \rangle  \alpha_\infty'(\tilde{u}) \, d\tilde{u}
\]
for $j \in \{1, \ldots, n-1\}$, where  $\{ T_\infty(0), \nu_\infty^1(0), \ldots, \nu_\infty^{n-1}(0)\}$ is an orthonormal basis at $\alpha_\infty(0)$  because of the strong convergence in \eqref{conv-1} and \eqref{conv-2}.
Then we can check as in \eqref{frame} that $\{ T_\infty(u), \nu_\infty^1(u), \ldots, \nu_\infty^{n-1}(u)\}$ is an orthonormal frame at $\alpha_\infty(u)$ for all $u \in S^1$. 

Now we have
\[
B_i:= \beta \bigl[ \nu_{i}^1\, \nu_{i}^2 \ldots \nu_{i}^{n-1} \bigr], \qquad 
W_i := \int^1_0 B_i B_i^T \, du, 
\]
and let $\eta_i$ be a unit eigenvector corresponding to the smallest eigenvalue  $\mu_i$ of $W_i$.
By \eqref{conv-2} we see that 
\[
W_i \to W_\infty := \int^1_0 B_\infty B_\infty^T \, du
\] 
with  
\[
B_\infty := \beta \bigl[ \nu_{\infty}^1 \nu_{\infty}^2 \ldots \nu_{\infty}^{n-1} \bigr]. 
\]
Thus we find $\mu_\infty \in \mathbb{R}$ such that 
\[
\mu_i \to \mu_\infty  \quad \text{as} \quad i \to \infty, 
\]
and $\mu_\infty$ is the smallest eigenvalue of $W_\infty$. 
Let $\eta_\infty $ be the unit eigenvector corresponding to $\mu_\infty$,  
then as in \eqref{eq:1st-eigen}, we have 
\begin{equation}
\label{eq:1st-eigen-2}
\mu_\infty= \int^1_0 \beta^2 \Bigl ( 1 - \ip{\eta_\infty , T_\infty }^2 \Bigr ) \, du. 
\end{equation}
Repeating the argument following \eqref{eq:1st-eigen}, $\mu_\infty$ cannot be zero, and we have our contradiction to the assumption that there is no $C$ bounding $\norm{W_\alpha^{-1}}$.

\end{proof}

\end{appendix}

\bibliography{okabe-schrader.bib}

\begin{thebibliography}{10}

\bibitem{Adams:2003aa}
{\sc Adams, R.~A., and Fournier, J. J.~F.}
\newblock {\em Sobolev spaces}, second~ed., vol.~140 of {\em Pure and Applied
  Mathematics (Amsterdam)}.
\newblock Elsevier/Academic Press, Amsterdam, 2003.

\bibitem{blatt2022minimising}
{\sc Blatt, S., Hopper, C.~P., and Vorderobermeier, N.}
\newblock A minimising movement scheme for the p-elastic energy of curves.
\newblock {\em Journal of Evolution Equations 22}, 2 (2022), 1--25.

\bibitem{Brockett:1970aa}
{\sc Brockett, R.}
\newblock {\em Finite dimensional linear systems}.
\newblock Series in decision and control. Wiley, 1970.

\bibitem{Bruveris:2015aa}
{\sc Bruveris, M.}
\newblock Completeness properties of {S}obolev metrics on the space of curves.
\newblock {\em J. Geom. Mech. 7}, 2 (2015), 125--150.

\bibitem{chill2009willmore}
{\sc Chill, R., Fa{\v s}angov{\'a}, E., and Sch{\"a}tzle, R.}
\newblock {Willmore blowups are never compact}.
\newblock {\em Duke Mathematical Journal 147}, 2 (2009), 345 -- 376.

\bibitem{DLP_2017}
{\sc Dall'Acqua, A., Lin, C.-C., and Pozzi, P.}
\newblock A gradient flow for open elastic curves with fixed length and clamped
  ends.
\newblock {\em Ann. Sc. Norm. Super. Pisa Cl. Sci. (5) 17}, 3 (2017),
  1031--1066.

\bibitem{dall2020elastic}
{\sc Dall'Acqua, A., Lin, C.-C., and Pozzi, P.}
\newblock Elastic flow of networks: short-time existence result.
\newblock {\em Journal of Evolution Equations\/} (2020), 1--46.

\bibitem{DallAcqua:2016aa}
{\sc Dall'Acqua, A., Pozzi, P., and Spener, A.}
\newblock The {{\L}}ojasiewicz-{S}imon gradient inequality for open elastic
  curves.
\newblock {\em J. Differential Equations 261}, 3 (2016), 2168--2209.

\bibitem{DKS_2000}
{\sc Dziuk, G., Kuwert, E., and Sch\"{a}tzle, R.}
\newblock Evolution of elastic curves in {$\Bbb R^n$}: existence and
  computation.
\newblock {\em SIAM J. Math. Anal. 33}, 5 (2002), 1228--1245.

\bibitem{Feehan:2016aa}
{\sc Feehan, P. M.~N.}
\newblock Global existence and convergence of solutions to gradient systems and
  applications to {Y}ang-{M}ills gradient flow, 2016.

\bibitem{Feehan:2020aa}
{\sc Feehan, P. M.~N., and Maridakis, M.}
\newblock {\L}ojasiewicz-{S}imon gradient inequalities for analytic and
  {M}orse-{B}ott functions on {B}anach spaces.
\newblock {\em J. Reine Angew. Math. 765\/} (2020), 35--67.

\bibitem{Klingenberg:1995aa}
{\sc Klingenberg, W. P.~A.}
\newblock {\em Riemannian geometry}, second~ed., vol.~1 of {\em De Gruyter
  Studies in Mathematics}.
\newblock Walter de Gruyter \& Co., Berlin, 1995.

\bibitem{Knappmann:2021aa}
{\sc Knappmann, J., Schumacher, H., Steenebr{\"u}gge, D., and von~der Mosel,
  H.}
\newblock A speed preserving {H}ilbert gradient flow for generalized integral
  menger curvature, 2021.

\bibitem{Koiso_1996}
{\sc Koiso, N.}
\newblock On the motion of a curve towards elastica.
\newblock In {\em Actes de la {T}able {R}onde de {G}\'{e}om\'{e}trie
  {D}iff\'{e}rentielle ({L}uminy, 1992)}, vol.~1 of {\em S\'{e}min. Congr.}
  Soc. Math. France, Paris, 1996, pp.~403--436.

\bibitem{Lang:1999sf}
{\sc Lang, S.}
\newblock {\em Fundamentals of differential geometry}, vol.~191 of {\em
  Graduate Texts in Mathematics}.
\newblock Springer-Verlag, New York, 1999.

\bibitem{Langer-Singer_1985}
{\sc Langer, J., and Singer, D.~A.}
\newblock Curve straightening and a minimax argument for closed elastic curves.
\newblock {\em Topology 24}, 1 (1985), 75--88.

\bibitem{Langer-Singer_1987}
{\sc Langer, J., and Singer, D.~A.}
\newblock Curve-straightening in {R}iemannian manifolds.
\newblock {\em Ann. Global Anal. Geom. 5}, 2 (1987), 133--150.

\bibitem{Lin_2012}
{\sc Lin, C.-C.}
\newblock {$L^2$}-flow of elastic curves with clamped boundary conditions.
\newblock {\em J. Differential Equations 252}, 12 (2012), 6414--6428.

\bibitem{Linner_1989}
{\sc Linn\'{e}r, A.}
\newblock Some properties of the curve straightening flow in the plane.
\newblock {\em Trans. Amer. Math. Soc. 314}, 2 (1989), 605--618.

\bibitem{Linner:1991aa}
{\sc Linn{\'e}r, A.}
\newblock Curve-straightening in closed {E}uclidean submanifolds.
\newblock {\em Comm. Math. Phys. 138}, 1 (1991), 33--49.

\bibitem{Linner_2003}
{\sc Linn\'{e}r, A.}
\newblock Symmetrized curve-straightening.
\newblock {\em Differential Geom. Appl. 18}, 2 (2003), 119--146.

\bibitem{MPP_2021}
{\sc Mantegazza, C., Pluda, A., and Pozzetta, M.}
\newblock A survey of the elastic flow of curves and networks.
\newblock {\em Milan J. Math. 89}, 1 (2021), 59--121.

\bibitem{MP_2021}
{\sc Mantegazza, C., and Pozzetta, M.}
\newblock The {{\L}}ojasiewicz-{S}imon inequality for the elastic flow.
\newblock {\em Calc. Var. Partial Differential Equations 60}, 1 (2021), Paper
  No. 56, 17.

\bibitem{Michor:2006aa}
{\sc Michor, P.~W., and Mumford, D.}
\newblock Riemannian geometries on spaces of plane curves.
\newblock {\em J. Eur. Math. Soc. (JEMS) 8}, 1 (2006), 1--48.

\bibitem{Muller_2020}
{\sc M\"{u}ller, M.}
\newblock On gradient flows with obstacles and {E}uler's elastica.
\newblock {\em Nonlinear Anal. 192\/} (2020), 111676, 48.

\bibitem{Neuberger:2010aa}
{\sc Neuberger, J.~W.}
\newblock {\em Sobolev gradients and differential equations}, second~ed.,
  vol.~1670 of {\em Lecture Notes in Mathematics}.
\newblock Springer-Verlag, Berlin, 2010.

\bibitem{NO_2014}
{\sc Novaga, M., and Okabe, S.}
\newblock Curve shortening-straightening flow for non-closed planar curves with
  infinite length.
\newblock {\em J. Differential Equations 256}, 3 (2014), 1093--1132.

\bibitem{NP_2020}
{\sc Novaga, M., and Pozzi, P.}
\newblock A second order gradient flow of {$p$}-elastic planar networks.
\newblock {\em SIAM J. Math. Anal. 52}, 1 (2020), 682--708.

\bibitem{O_2007}
{\sc Okabe, S.}
\newblock The motion of elastic planar closed curves under the area-preserving
  condition.
\newblock {\em Indiana Univ. Math. J. 56}, 4 (2007), 1871--1912.

\bibitem{O_2008}
{\sc Okabe, S.}
\newblock The dynamics of elastic closed curves under uniform high pressure.
\newblock {\em Calc. Var. Partial Differential Equations 33}, 4 (2008),
  493--521.

\bibitem{OPW}
{\sc Okabe, S., Pozzi, P., and Wheeler, G.}
\newblock A gradient flow for the {$p$}-elastic energy defined on closed planar
  curves.
\newblock {\em Math. Ann. 378}, 1-2 (2020), 777--828.

\bibitem{OW_2021}
{\sc Okabe, S., and Wheeler, G.}
\newblock The $p$-elastic flow for planar closed curves with constant
  parametrization.
\newblock {\em arXiv:2104.03570\/} (2021).

\bibitem{Palais:1988fv}
{\sc Palais, R.~S., and Terng, C.-L.}
\newblock {\em Critical point theory and submanifold geometry}, vol.~1353 of
  {\em Lecture Notes in Mathematics}.
\newblock Springer-Verlag, Berlin, 1988.

\bibitem{P_1996}
{\sc Polden, A.}
\newblock Curves and surfaces of least total curvature and fourth-order flows.
\newblock {\em PhD Thesis, Universit\"at T\"ubingen\/} (1996).

\bibitem{Reiter-Schumacher_2020}
{\sc Reiter, P., and Schumacher, H.}
\newblock Sobolev gradients for the {M}{\"o}bius energy.
\newblock {\em Archive for Rational Mechanics and Analysis 242}, 2 (2021),
  701--746.

\bibitem{rupp2020lojasiewicz}
{\sc Rupp, F.}
\newblock On the {{\L}}ojasiewicz--{S}imon gradient inequality on submanifolds.
\newblock {\em Journal of Functional Analysis 279}, 8 (2020), 108708.

\bibitem{RS_2020}
{\sc Rupp, F., and Spener, A.}
\newblock Existence and convergence of the length-preserving elastic flow of
  clamped curves.
\newblock {\em arXiv:2009.06991\/} (2020).

\bibitem{Schrader:2016aa}
{\sc Schrader, P.}
\newblock {\em Global analysis of one-dimensional variational problems}.
\newblock PhD thesis, The University of Western Australia, 2016.

\bibitem{Schrader:2016ab}
{\sc Schrader, P.}
\newblock Morse theory for elastica.
\newblock {\em J. Geom. Mech. 8}, 2 (2016), 235--256.

\bibitem{Schrader:2021aa}
{\sc Schrader, P., Wheeler, G., and Wheeler, V.-M.}
\newblock On the ${H}^1(ds)$-gradient flow for the length functional.
\newblock {\em arXiv:2102.07305\/} (2021).

\bibitem{Spener_2017}
{\sc Spener, A.}
\newblock Short time existence for the elastic flow of clamped curves.
\newblock {\em Math. Nachr. 290}, 13 (2017), 2052--2077.

\bibitem{Wen_1993}
{\sc Wen, Y.}
\newblock {$L^2$} flow of curve straightening in the plane.
\newblock {\em Duke Math. J. 70}, 3 (1993), 683--698.

\bibitem{Wen_1995}
{\sc Wen, Y.}
\newblock Curve straightening flow deforms closed plane curves with nonzero
  rotation number to circles.
\newblock {\em J. Differential Equations 120}, 1 (1995), 89--107.

\bibitem{Whittlesey:1965aa}
{\sc Whittlesey, E.~F.}
\newblock Analytic functions in {B}anach spaces.
\newblock {\em Proc. Amer. Math. Soc. 16\/} (1965), 1077--1083.

\bibitem{Zeidler:1986aa}
{\sc Zeidler, E.}
\newblock {\em Nonlinear functional analysis and its applications. {I}}.
\newblock Springer-Verlag, New York, 1986.
\newblock Fixed-point theorems, Translated from the German by Peter R. Wadsack.

\end{thebibliography}
\bibliographystyle{acm}


\end{document}